\documentclass[a4paper,10pt]{amsart}
\usepackage[utf8]{inputenc}
\usepackage{amsmath, amssymb, amsthm, graphicx, wrapfig, wasysym, fullpage}

\newtheorem{Prop}{Proposition}[section]
\newtheorem{Thm}[Prop]{Theorem}
\newtheorem{Lemma}[Prop]{Lemma}
\newtheorem{Cor}[Prop]{Corollary}
\newtheorem{Def}[Prop]{Definition}
\theoremstyle{remark}

\numberwithin{equation}{section}

\title{Combinatorics of hexagonal fully packed loop configurations}
\author{Sabine Beil}
\address{Sabine Beil, Universit\"at Wien, Fakult\"at f\"ur Mathematik, Oskar-Morgenstern-Platz~1, 1090 Wien, Austria}
\email{sabine.beil@univie.ac.at}
\thanks{Supported by the Austrian Science Foundation FWF, START grant Y463.}

\begin{document}

\begin{abstract}
In this article, fully packed loop configurations of hexagonal shape (HFPLs) are defined. They generalize triangular fully packed loop configurations. To encode the boundary conditions of an HFPL, a sextuple 
$(\mathsf{l}_\mathsf{T},\mathsf{t},\mathsf{r}_\mathsf{T};\mathsf{r}_\mathsf{B},\mathsf{b},\mathsf{l}_\mathsf{B})$ of $01$-words is assigned to it.
In the first main result of this article, necessary conditions for the boundary $(\mathsf{l}_\mathsf{T},\mathsf{t},\mathsf{r}_\mathsf{T};\mathsf{r}_\mathsf{B},\mathsf{b},\mathsf{l}_\mathsf{B})$ of an HFPL
are stated. For instance, 
the inequality $d(\mathsf{r}_\mathsf{B})+d(\mathsf{b})+d(\mathsf{l}_\mathsf{B})\geq d(\mathsf{l}_\mathsf{T})+d(\mathsf{t})+d(\mathsf{r}_\mathsf{T})+\vert\mathsf{l}_\mathsf{T}\vert_1\vert\mathsf{t}\vert_0+\vert\mathsf{t}\vert_1
\vert\mathsf{r}_\mathsf{T}\vert_0+\vert\mathsf{r}_\mathsf{B}\vert_0\vert\mathsf{l}_\mathsf{B}\vert_1$ has to be fulfilled, where 
$\vert\cdot\vert_i$ denotes the number of occurrences of $i$ for $i=0,1$ and $d(\cdot)$ denotes the number of 
inversions. 
The other main contribution of this article is the enumeration of HFPLs with boundary $(\mathsf{l}_\mathsf{T},\mathsf{t},\mathsf{r}_\mathsf{T};\mathsf{r}_\mathsf{B},\mathsf{b},\mathsf{l}_\mathsf{B})$ such that 
$d(\mathsf{r}_\mathsf{B})+d(\mathsf{b})+d(\mathsf{l}_\mathsf{B})-d(\mathsf{l}_\mathsf{T})-d(\mathsf{t})-d(\mathsf{r}_\mathsf{T})-\vert\mathsf{l}_\mathsf{T}\vert_1\vert\mathsf{t}\vert_0-
\vert\mathsf{t}\vert_1\vert\mathsf{r}_\mathsf{T}\vert_0-\vert\mathsf{r}_\mathsf{B}\vert_0\vert\mathsf{l}_\mathsf{B}\vert_1=0,1$. To be more precise, in the first case they are enumerated by
Littlewood-Richardson coefficients and in the second case their number is expressed in terms of Littlewood-Richardson coefficients. 
\end{abstract}

\maketitle

\section*{Introduction}
Fully packed loop configurations (FPLs) came up in statistical mechanics. Later, it turned out that they are in bijection to alternating sign matrices. Thus, FPLs are enumerated by the famous formula
for alternating sign matrices in \cite{Zeilberger}.
In a natural way, every FPL defines a (non-crossing) matching of the occupied external edges -- the so-called \textit{link pattern} $\pi$ --  by matching those which are joined by a path. Crucial in the development
of triangular fully packed loop configurations (TFPLs) are FPLs corresponding to a link pattern with a large number of nested arches: they admit a combinatorial decomposition in which TFPLs naturally arise. The previous 
came up in the course of the proof in \cite{CKLN} of a conjecture in \cite{Zuber}. It states that the number of FPLs corresponding to a given link pattern with $m$ nested arches is a polynomial in $m$.
In the course of the study of FPLs corresponding to link patterns with a large number of nested arches, some first combinatorics of TFPLs were derived, see \cite{Nadeau1} and \cite{Thapper}. For example, necessary 
conditions for the boundary $(u,v;w)$ -- a triple of $01$-words -- of a TFPL came up in \cite{CKLN}, \cite{TFPL} and \cite{Nadeau1}. One of these conditions states that $d(w)\geq d(u)+d(v)$ where  
$d(\cdot)$ denotes the number of inversions of a $01$-word. It was given in \cite{Thapper} in the Dyck word case. 
Later, another proof in connection with TFPLs together with an orientation of the edges was given in~\cite{TFPL}.
Moreover, for oriented TFPLs an interpretation of the difference $d(w)-d(u)-d(v)$ in terms of occurrences of local configurations is proven in \cite{TFPL}. 
This point of view turned out fruitful: under the constraint that $d(w)-d(u)-d(v)=0$, oriented TFPLs with boundary $(u,v;w)$
are enumerated by the Littlewood-Richardson coefficient $c_{\lambda(u),\lambda(v)}^{\lambda(w)}$ where $\lambda(\cdot)$ denotes the Young diagram corresponding to a $01$-word. First, this was only shown for
Dyck words $u,v,w$ in \cite{Nadeau2}. Later, it was extended to all ordinary and to oriented TFPLs with boundary $(u,v;w)$ in \cite{TFPL}. More precisely, a bijection between oriented TFPLs with boundary $(u,v;w)$ and
Knutson-Tao puzzles with boundary $(u,v;w)$ was constructed. Finally, the number of ordinary as well as of oriented TFPLs with boundary $(u,v;w)$ such that $d(w)-d(u)-d(v)=1$ was expressed in terms of 
Littlewood-Richardson coefficients in \cite{TFPL}.\\

The goal of this article is to generalize the results for triangular fully packed loop configurations (TFPLs) to fully packed loop configurations of hexagonal shape (HFPLs). In Section~\ref{Sec:HFPLs}, 
HFPLs of size $(K,L,M,N)$ are defined. 
Examples of a TFPL and of an HFPL are given in Figure~\ref{Fig:Example_TFPL_HFPL}. Indeed, HFPLs generalize TFPLs: TFPLs of size $n$ appear as subsets of HFPLs when considered HFPLs of size $(n,0,n,0)$.
\begin{figure}[tbh]
\includegraphics[width=.8\textwidth]{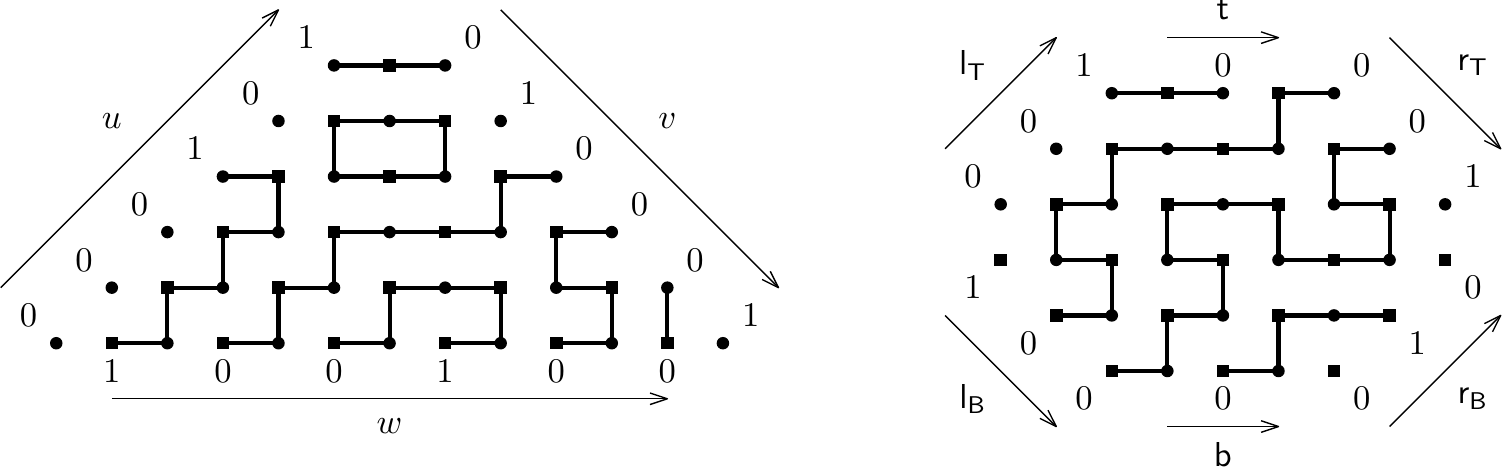}
\caption{Examples of a TFPL of size $6$ (left) and an HFPL of size $(3,1;3,3)$ (right). The boundary of the TFPL is $(000101,010001;100100)$ and the boundary of the HFPL is $(001,0,001;010,0,100)$.}
\label{Fig:Example_TFPL_HFPL}
\end{figure}
Furthermore, to each HFPL a sextuple $(\mathsf{l}_\mathsf{T},\mathsf{t},\mathsf{r}_\mathsf{T};\mathsf{r}_\mathsf{B},\mathsf{b},\mathsf{l}_\mathsf{B})$
of $01$-words of lengths $(K,L,M;N,K+L-N,M+N-K)$ respectively, that encodes the boundary conditions, is assigned. 
For instance, the boundary of the HFPL depicted in Figure~\ref{Fig:Example_TFPL_HFPL} is $(001,0,001;010,0,100)$.
The boundary of an HFPL generalizes the boundary of a TFPL in the following way: a TFPL with boundary $(u,v;w)$, when considered an HFPL of size $(n,0,n,0)$, has boundary 
$(u,\varepsilon,v;\varepsilon,w,\varepsilon)$, where $\varepsilon$ denotes the empty word.\\

The existence of an HFPL with boundary
$(\mathsf{l}_\mathsf{T},\mathsf{t},\mathsf{r}_\mathsf{T};\mathsf{r}_\mathsf{B},\mathsf{b},\mathsf{l}_\mathsf{B})$
implies the following constraints on $\mathsf{l}_\mathsf{T},\mathsf{t},\mathsf{r}_\mathsf{T},\mathsf{r}_\mathsf{B},\mathsf{b},\mathsf{l}_\mathsf{B}$:

\begin{enumerate}
 \item $\vert\mathsf{l}_\mathsf{T}\vert_0+\vert\mathsf{t}\vert_0=\vert\mathsf{r}_\mathsf{B}\vert_0+\vert\mathsf{b}\vert_0$;
 \item $\vert\mathsf{t}\vert_1+\vert\mathsf{r}_\mathsf{T}\vert_1=\vert\mathsf{b}\vert_1+\vert\mathsf{l}_\mathsf{B}\vert_1$;
 \item $\mathsf{l}_\mathsf{T}\,\mathsf{t}\leq\mathsf{b}\,\mathsf{r}_\mathsf{B}$ and $\mathsf{t}\,\mathsf{r}_\mathsf{T}\leq\mathsf{l}_\mathsf{B}\,\mathsf{b}$ for the concatenations $\mathsf{l}_\mathsf{T}\,\mathsf{t}$,
 $\mathsf{b}\,\mathsf{r}_\mathsf{B}$, $\mathsf{t}\,\mathsf{r}_\mathsf{T}$ and $\mathsf{l}_\mathsf{B}\,\mathsf{b}$;
 \item $d(\mathsf{r}_\mathsf{B})+d(\mathsf{b})+d(\mathsf{l}_\mathsf{B})\geq d(\mathsf{l}_\mathsf{T})+d(\mathsf{t})+d(\mathsf{r}_\mathsf{T})+\vert\mathsf{l}_\mathsf{T}\vert_1\vert\mathsf{t}\vert_0+
\vert\mathsf{t}\vert_1\vert\mathsf{r}_\mathsf{T}\vert_0+\vert\mathsf{r}_\mathsf{B}\vert_0\vert\mathsf{l}_\mathsf{B}\vert_1$.
\end{enumerate}
Here, $\vert\omega\vert_i$ denotes the number of occurrences of $i$ in a $01$-word $\omega$ for $i=0,1$ and $d(\omega)$ denotes the number of inversions in $\omega$, i.e., of pairs $k<\ell$ such that 
$\omega_k>\omega_\ell$. The previous constraints are the content of Theorem~\ref{Thm:NecCondHFPL}. They generalize the following constraints on the boundary $(u,v;w)$ of a TFPL: 
(1) $\vert u\vert_0=\vert v\vert_0=\vert w\vert_0$; (2) $u\leq w$ and $v\leq w$; (3) $d(w)\geq d(u)+d(v)$.  \\

The crucial idea for the proofs in this article is to add an orientation to each edge of an HFPL. This is done by generalizing the way an orientation is added to each edge of a TFPL in \cite{TFPL}.
The interplay between ordinary and oriented HFPLs is content of Section~\ref{Sec:Recovering_HFPLs_from_oriented_HFPLs}. It will be shown that the set of ordinary HFPLs with boundary 
$(\mathsf{l}_\mathsf{T},\mathsf{t},\mathsf{r}_\mathsf{T},\mathsf{r}_\mathsf{B},\mathsf{b},\mathsf{l}_\mathsf{B})$ can be regarded as a subset of the 
set of oriented HFPLs with boundary $(\mathsf{l}_\mathsf{T},\mathsf{t},\mathsf{r}_\mathsf{T},\mathsf{r}_\mathsf{B},\mathsf{b},\mathsf{l}_\mathsf{B})$. In particular, the existence of an HFPL with
boundary $(\mathsf{l}_\mathsf{T},\mathsf{t},\mathsf{r}_\mathsf{T},\mathsf{r}_\mathsf{B},\mathsf{b},\mathsf{l}_\mathsf{B})$ implies the existence of an oriented HFPL with boundary
$(\mathsf{l}_\mathsf{T},\mathsf{t},\mathsf{r}_\mathsf{T},\mathsf{r}_\mathsf{B},\mathsf{b},\mathsf{l}_\mathsf{B})$.
On the other hand, a weighted enumeration of oriented HFPLs is introduced, from which the number of ordinary HFPLs can 
be derived. The latter is specified in Corollary~\ref{Cor:WeightedEnumeration_HFPL}. This weighted enumeration of oriented HFPLs generalizes the weighted enumeration of oriented TFPLs in \cite{TFPL}.\\

In Section~\ref{Sec:Combinatorial_Interpretation_HFPL}, an interpretation of the integer 
\begin{equation}
exc(\mathsf{l}_\mathsf{T},\mathsf{t},\mathsf{r}_\mathsf{T};\mathsf{r}_\mathsf{B},\mathsf{b},\mathsf{l}_\mathsf{B})=
d(\mathsf{r}_\mathsf{B})+d(\mathsf{b})+d(\mathsf{l}_\mathsf{B})-d(\mathsf{l}_\mathsf{T})-d(\mathsf{t})-d(\mathsf{r}_\mathsf{T})-\vert\mathsf{l}_\mathsf{T}\vert_1\vert\mathsf{t}\vert_0-
\vert\mathsf{t}\vert_1\vert\mathsf{r}_\mathsf{T}\vert_0-\vert\mathsf{r}_\mathsf{B}\vert_0\vert\mathsf{l}_\mathsf{B}\vert_1
\notag\end{equation}
in terms of numbers of occurrences of certain local configurations in an oriented HFPL with boundary 
$(\mathsf{l}_\mathsf{T},\mathsf{t},\mathsf{r}_\mathsf{T},\mathsf{r}_\mathsf{B},\mathsf{b},\mathsf{l}_\mathsf{B})$ is proven.
The thereby counted local configurations resemble the local configurations that are counted by $d(w)-d(u)-d(v)$ in an oriented TFPLs with 
boundary $(u,v;w)$, see \cite{TFPL}. 
From the interpretation of $exc(\mathsf{l}_\mathsf{T},\mathsf{t},\mathsf{r}_\mathsf{T};\mathsf{r}_\mathsf{B},\mathsf{b},\mathsf{l}_\mathsf{B})$ in terms of numbers of occurrences of local configurations in an 
oriented HFPL with boundary $(\mathsf{l}_\mathsf{T},\mathsf{t},\mathsf{r}_\mathsf{T};\mathsf{r}_\mathsf{B},\mathsf{b},\mathsf{l}_\mathsf{B})$, the last of the stated constraints on the
boundary of an HFPL follows immediately. To give nice proofs of that interpretation and also of the other 
constraints on the boundary of an HFPL, oriented HFPLs are encoded by path-tangles. 
Path-tangles and the bijection between them and oriented HFPLs are treated in Section~\ref{Subsec:PathTanglesHFPL}. \\

In the last section, HFPLs with boundary $(\mathsf{l}_\mathsf{T},\mathsf{t},\mathsf{r}_\mathsf{T},\mathsf{r}_\mathsf{B},\mathsf{b},\mathsf{l}_\mathsf{B})$ where 
$exc(\mathsf{l}_\mathsf{T},\mathsf{t},\mathsf{r}_\mathsf{T};\mathsf{r}_\mathsf{B},\mathsf{b},\mathsf{l}_\mathsf{B})\in\{0,1\}$ are considered. In the case when 
$exc(\mathsf{l}_\mathsf{T},\mathsf{t},\mathsf{r}_\mathsf{T};\mathsf{r}_\mathsf{B},\mathsf{b},\mathsf{l}_\mathsf{B})=0$ both ordinary and oriented HFPLs with boundary 
$(\mathsf{l}_\mathsf{T},\mathsf{t},\mathsf{r}_\mathsf{T};\mathsf{r}_\mathsf{B},\mathsf{b},\mathsf{l}_\mathsf{B})$ are enumerated by the Littlewood-Richardson coefficient
\begin{equation}
c_{\lambda(\textbf{0}_{\vert\mathsf{l}_\mathsf{B}\vert_0}\,\textbf{1}_{\vert\mathsf{l}_\mathsf{B}\vert_1}\,\mathsf{l}_\mathsf{T}\,\mathsf{t}),\lambda(\textbf{0}_{\vert\mathsf{t}\vert_0}\,
\textbf{1}_{\vert\mathsf{t}\vert_1}\,\mathsf{r}_\mathsf{T}\,\textbf{0}_{\vert\mathsf{r}_\mathsf{B}\vert_0}\,\textbf{1}_{\vert\mathsf{r}_\mathsf{B}\vert_1})}^{\lambda(\mathsf{l}_\mathsf{B}\,\mathsf{b}\,
\mathsf{r}_\mathsf{B})}.
\label{Introduction:LR_coefficient}
\end{equation}
Here, $\textbf{0}_n$ (respectively $\textbf{1}_n$) denotes the word of length $n$ made up solely of zeroes (respectively of ones). 
The enumeration of both ordinary and oriented HFPLs with boundary $(\mathsf{l}_\mathsf{T},\mathsf{t},\mathsf{r}_\mathsf{T};\mathsf{r}_\mathsf{B},\mathsf{b},\mathsf{l}_\mathsf{B})$, where 
exc$(\mathsf{l}_\mathsf{T},\mathsf{t},\mathsf{r}_\mathsf{T};\mathsf{r}_\mathsf{B},\mathsf{b},\mathsf{l}_\mathsf{B})=0$, by the Littlewood-Richardson coefficient in (\ref{Introduction:LR_coefficient}) generalizes the 
enumeration of both ordinary and oriented TFPLs with boundary $(u,v;w)$, where $d(w)=d(u)+d(v)$, by the 
Littlewood-Richardson coefficient $c_{\lambda(u),\lambda(v)}^{\lambda(w)}$.  
Finally, the number of HFPLs with boundary 
$(\mathsf{l}_\mathsf{T},\mathsf{t},\mathsf{r}_\mathsf{T};\mathsf{r}_\mathsf{B},\mathsf{b},\mathsf{l}_\mathsf{B})$ where 
$exc(\mathsf{l}_\mathsf{T},\mathsf{t},\mathsf{r}_\mathsf{T};\mathsf{r}_\mathsf{B},\mathsf{b},\mathsf{l}_\mathsf{B})=1$ is expressed in terms of Littlewood-Richardson coefficients. 
This expression generalizes the expression of the number of TFPLs with boundary $(u,v;w)$ where $d(w)-d(u)-d(v)=1$ in terms of Littlewood-Richardson coefficients given in \cite{TFPL}.

\section{Preliminaries}
\subsection{Words}

In connection with hexagonal fully packed loop configurations, words play an important role. 
When it is spoken of a \textit{word} $\omega$ of length $n$ it is referred to a finite sequence $\omega=\omega_1\omega_2\cdots \omega_n$ where $\omega_i\in \{0,1\}$ for all 
$1\leq i\leq n$. Given a word $\omega$ the number of occurrences of 0 (resp. 1) in $\omega$ is denoted by $\vert \omega\vert_0$ (resp. $\vert \omega\vert_1$). 
Furthermore, it is said that two words $\omega, \sigma$ of length $n$ with the same number of occurrences of ones satisfy $\omega\leq\sigma$ if 
$\vert\omega_1\cdots\omega_m\vert_1\leq\vert\sigma_1\cdots\sigma_m\vert_1$ holds for all $1\leq m\leq n$. Finally, the number of inversions of $\omega$, i.e., pairs
$1\leq i<j\leq n$ that satisfy $\omega_i=1$ and $\omega_j=0$, is denoted by $d(\omega)$.\\

It is a well known fact that words are in bijection with Young diagrams. Throughout this article, the following bijection $\lambda$ from the set of words to the set of Young diagrams is chosen:
to a given word $\omega$ a path on the square lattice is constructed by drawing a $(0,1)$-step if $\omega_i=0$ and a $(1,0)$-step if $\omega_i=1$ for i from $1$ to $n$. Additionally, 
 a horizontal line up the path's starting point and vertical line to the left of its ending point are drawn. The resulting region then encloses a Young diagram, which shall be the image of $w$ under $\lambda$. In 
Figure~\ref{Fig:BijectionWordsYoung}, an example of a word and its corresponding Young diagram is given.  
\begin{figure}[tbh]
\centering
\includegraphics[width=.15\textwidth]{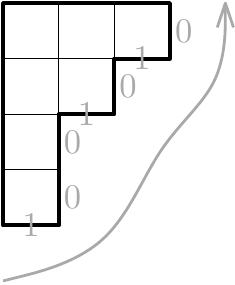}
\caption{The Young diagram $\lambda(1001010)$.}
\label{Fig:BijectionWordsYoung}
\end{figure}
The number of columns (resp. rows) of $\lambda(\omega)$ equals $\vert\omega\vert_1$ (resp. $\vert\omega\vert_0$).
Furthermore, for two words $\omega$ and $\sigma$ of length $n$ it holds $\omega\leq \sigma$ if and only if $\lambda(\omega)$ is contained in $\lambda(\sigma)$. 
Finally, the number of cells of $\lambda(\omega)$ is given by $d(\omega)$.\\

Given a word $\omega=\omega_1\cdots\omega_n$ the word $\overleftarrow{\omega}$ is defined as the word $\omega_n\cdots\omega_1$, the word $\overline{\omega}$ is defined as the word
$\overline{\omega_1}\cdots\overline{\omega_n}$ where $\overline{0}=1$ and $\overline{1}=0$ and $\omega^\ast$ is defined as the word $\overline{\overleftarrow{\omega}}$. \\

In the next subsection, \textit{Dyck words} come up: a Dyck word is a word $\omega$ of even length such that $\vert\omega\vert_1=\vert\omega\vert_0$
and each prefix $\omega'$ of $\omega$ satisfies $\vert\omega'\vert_0\geq\vert\omega'\vert_1$.

\subsection{Extended link patterns}
A \textit{link pattern} $\pi$ of size $2n$ is defined as a partition of $\{1,2,\dots,2n\}$ into $n$ blocks of size $2$ that are pairwise non-crossing, i.e., there are no integers $i<j<k<l$ such that $\{i,k\}$ and 
$\{j,l\}$ are both in $\pi$. In the following, link patterns are represented by non-crossing arches between $2n$ aligned points. An example of a link pattern is given in Figure~\ref{Fig:LinkPattern}.
It is a well known fact that link patterns of size $2n$ are in bijection with Dyck words of length $2n$: to a link pattern $\pi$ of size $2n$ the Dyck word $\omega$ of length $2n$ is assigned where 
$\omega_i=0$ and $\omega_j=1$ for each pair $\{i,j\}$ in $\pi$ with $i<j$. 
For example, the Dyck word corresponding to the link pattern depicted in Figure~\ref{Fig:LinkPattern} is \hbox{$001101000111$}. In this article, a more general notion of link patterns is needed. 

\begin{figure}[tbh]
\includegraphics[width=.4\textwidth]{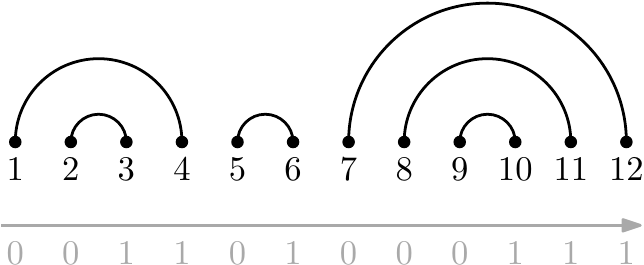}
\caption{The link pattern $\{\{1,4\},\{2,3\},\{5,6\},\{7,12\},\{8,11\},\{9,10\}\}$.}
\label{Fig:LinkPattern}
\end{figure}

\begin{Def}
An \textnormal{extended link pattern} $\pi$ on $\{1,\dots,n\}$ is the data of integers 
\begin{equation}1\leq \ell_1<\ell_2<\cdots<\ell_i< r_1<r_2<\cdots<r_j\leq n\notag\end{equation}                                                                                 
together with a link pattern
on each maximal interval of integers in $\{1,\dots,n\}$ that does not contain any of the points $\ell_k$ or $r_k$. The integers $\ell_1,\ell_2,\dots,\ell_i$ are said to be the \textnormal{left points} of $\pi$ and the 
integers $r_1,r_2,\dots,r_j$ are said to be the \textnormal{right points} of $\pi$. 
\end{Def}

In the figures, a left point $\ell_k$ of an extended link pattern $\pi$ is represented by attaching the extremity of an arch to the point $\ell_k$, with the arch going left, whereas a right point $r_k$ is represented
by attaching the extremity of an arch to the point $r_k$, with the arch going right.
An example of an extended link pattern is given in Figure~\ref{Fig:ExtendedLinkPattern}.
To an extended link pattern $\pi$ with left points $\ell_1<\ell_2<\cdots<\ell_i$ and right points $r_1<r_2<\cdots<r_j$ a word $\omega=\textbf{w}(\pi)$ is assigned as follows: as a start
it is set $\omega_{\ell_k}=1$ for all $1\leq k\leq i$ and $\omega_{r_k}=0$ for all $1\leq k\leq j$. Then each link pattern in $\pi$ is associated with its corresponding Dyck word. For instance, the word assigned
to the extended link pattern in Figure~\ref{Fig:ExtendedLinkPattern} is $1001101101010$.

\begin{figure}[tbh]
\includegraphics[width=.5\textwidth]{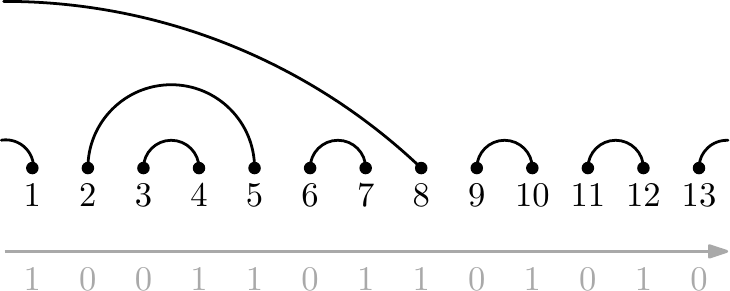}
\caption{An extended link pattern with left points $1$ and $8$ and right point $13$.}
\label{Fig:ExtendedLinkPattern}
\end{figure}

\begin{Prop}\label{Prop:BijectionWordsExtendedLinkPatterns}
The map \textnormal{\textbf{w}} is a bijection from the set of extended link patterns on $\{1,\dots,n\}$ to the set of words of length $n$. 
\end{Prop}

The previous proposition is given in \cite[Proposition~1.6]{TFPL}. At some point in this article, it will become necessary to consider extended link patterns together with an orientation of the arches:

\begin{Def}\begin{enumerate}
\item A \textnormal{directed extended link pattern} $\overrightarrow{\pi}$ on $\{1,\dots,n\}$ is an extended link pattern $\pi$ on $\{1,\dots,n\}$ 
where each $i$ in $\{1,\dots,n\}$ is either a \textnormal{sink} or a \textnormal{source} with the constraint that
each pair in one of the link patterns in $\pi$ consists of a sink and a source. Furthermore, the set of pairs in $\overrightarrow{\pi}$ where the larger integer is the source is denoted by
$RL(\overrightarrow{\pi})$. To a directed extended link pattern $\overrightarrow{\pi}$ on $\{1,\dots,n\}$ its \textnormal{source-sink-word} $w=w_1\cdots w_n$ is assigned as follows: for each $i$ from $1$ to $n$ set
$w_i=0$ if $i$ is a source or $w_i=1$ if $i$ is a sink.
\item A directed extended link pattern is \textnormal{left-hand-incoming} if all left points are sinks, respectively \textnormal{right-hand-outgoing} if all right points are sources.
\end{enumerate}
\end{Def}

A directed extended link pattern is represented by an extended link pattern together with an orientation of each arch and half arch such that each arch attached to a source is outgoing and each 
arch attached to a sink is incoming.
An example of a right-hand-outgoing directed extended link pattern is depicted in Figure~\ref{Fig:DirectedExtendedLinkPattern}. Its source-sink-word is $1010101010010$.

\begin{figure}[tbh]
\includegraphics[width=.5\textwidth]{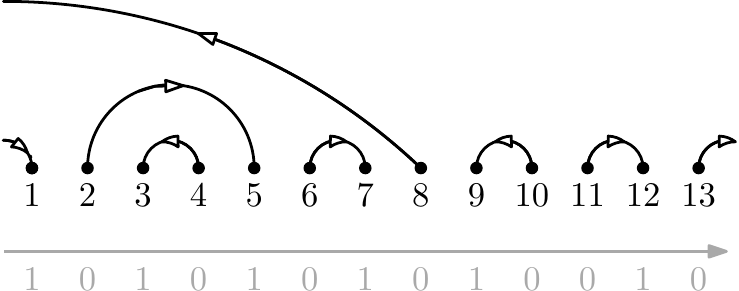}
\caption{A directed extended link pattern with sources $2,4,6,8,10,11,13$ and sinks $1,3,5,7,9,12$.}
\label{Fig:DirectedExtendedLinkPattern}
\end{figure}

\begin{Def}
\begin{enumerate}
\item A word $w'$ of size $n$ is \textnormal{feasible} for a word $w$ of length $n$ if there exists a directed extended link pattern $\overrightarrow{\pi}$ with underlying extended link pattern 
\textnormal{\textbf{w}}$^{-1}(w')$ such that $w$ is the source-sink-word of $\overrightarrow{\pi}$. Such a $\overrightarrow{\pi}$ is unique and therefore one can define $g(w,w')=RL(\overrightarrow{\pi})$ for
all words $w,w'$ such that $w'$ is feasible for $w$. 
\item A word $w'$ feasible for a word $w$ is said to be \textnormal{left-points-fixing} if $\overrightarrow{\pi}$ is left-hand-incoming, respectively \textnormal{right-points-fixing} if 
$\overrightarrow{\pi}$ is right-hand-outgoing.
\end{enumerate}  
\end{Def}

For instance, the word $1001101101010$ is feasible for $1010101010010$: the latter is the source-sink word of the directed extended link pattern in Figure~\ref{Fig:DirectedExtendedLinkPattern} and 
the former corresponds to the extended link pattern in Figure~\ref{Fig:ExtendedLinkPattern}. Thus, $g(1010101010010,1001101101010)=2$. If a word $w'$ is left-points-fixing feasible for a word $w$, then
$w_i=1$ for each left point $i$ of $\overrightarrow{\pi}$. On the other hand, if $w'$ is right-points-fixing feasible for $w$, then $w_i=0$ for each right point $i$ of $\overrightarrow{\pi}$.

\section{Hexagonal fully packed loop configurations}\label{Sec:HFPLs}

In this section the main objects of this article are introduced, namely hexagonal FPLs. From now on, let $K,L,M$ and $N$ be non-negative integers such that $K\leq M+N$ and $N\leq K+L$. 

\begin{Def}
The graph $H^{K,L,M,N}$ is defined as the induced subgraph of the square lattice with vertex set 
\begin{equation}
\{(x,y)\in\mathbb{Z}^2: y\leq x, y\leq K-1, y\leq -x+2(K+L), y\geq -x-1, y\geq -M-N+K, y\geq x-2(M+L)-1\}.
\notag\end{equation}
\end{Def}

\begin{figure}[tbh]
\includegraphics[width=.6\textwidth]{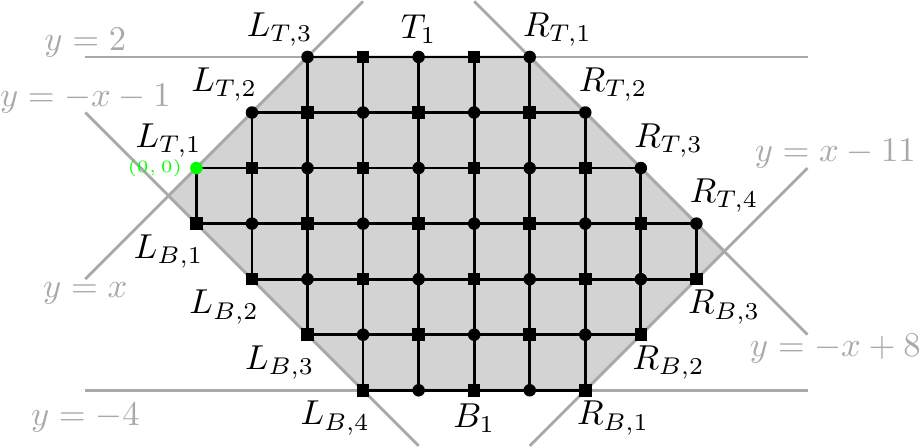}
\caption{The graph $H^{3,1,4,3}$.}
\label{Fig:H_3_1_4_3}
\end{figure}

In Figure~\ref{Fig:H_3_1_4_3}, the graph $H^{3,1,4,3}$ is depicted. As already indicated in Figure~\ref{Fig:H_3_1_4_3}, the vertices of $H^{K,L,M,N}$ are partitioned into \textit{odd} and \textit{even} vertices
in a chessboard manner such that the leftmost vertex of the top row of $H^{K,L,M,N}$ is odd.
In the pictures, odd vertices are illustrated by circles and even vertices by squares. Some vertices of $H^{K,L,M,N}$ are of special interest: 
\begin{itemize}
\item the leftmost vertices $\mathcal{L}_T=\{L_{T,1},\dots, L_{T,K}\}$ of each of the top $K$ rows of $H^{K,L,M,N}$;
\item the rightmost vertices $\mathcal{R}_T=\{R_{T,1},\dots, R_{T,M}\}$ of each of the top $M$ rows of $H^{K,L,M,N}$;
\item the odd vertices $\mathcal{T}=\{T_1,\dots,T_L\}$ of the top row of $H^{K,L,M,N}$ that are not in $\mathcal{L}_T\cup\mathcal{R}_T$;
\item the leftmost vertices $\mathcal{L}_B=\{L_{B,1},\dots, L_{B,M+N-K}\}$ of each of the bottom $M+N-K$ rows of $H^{K,L,M,N}$;
\item the rightmost vertices $\mathcal{R}_B=\{R_{B,1},\dots, R_{B,N}\}$ of each of the bottom $N$ rows of $H^{K,L,M,N}$;
\item the even vertices $\mathcal{B}=\{B_1,\dots,B_{K+L-N}\}$ of the bottom row of $H^{K,L,M,N}$, that are not in $\mathcal{L}_B\cup\mathcal{R}_B$.
\end{itemize}
All vertices are numbered from left to right.

\subsection{Hexagonal fully packed loop configurations}

\begin{Def}\label{Def:HFPL}
A \textnormal{hexagonal fully packed loop configuration} (HFPL) of size $(K,L,M,N)$ is a subgraph $f$ of $H^{K,L,M,N}$ that satisfies:
 \begin{enumerate}
  \item The vertices in $\mathcal{L}_B\cup\mathcal{L}_T\cup\mathcal{R}_B\cup\mathcal{R}_T$ are either of degree $0$ or of degree $1$.
  \item The vertices in $\mathcal{T}\cup\mathcal{B}$ are of degree $1$.
  \item All other vertices of $H^{K,L,M,N}$ are of degree $2$.
  \item There is neither a path in $f$ that joins two vertices in $\mathcal{L}_B\cup\mathcal{L}_T$ nor a path that joins two vertices in $\mathcal{R}_B\cup\mathcal{R}_T$.
 \end{enumerate}
\end{Def}

\begin{figure}[tbh]
\includegraphics[width=.4\textwidth]{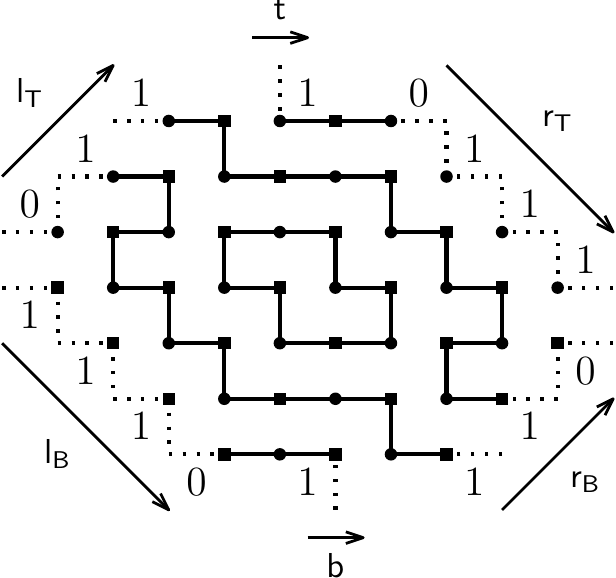}
\caption{A hexagonal fully packed loop configuration of size $(3,1,4,3)$.}
\label{Fig:Hexagonal_FPL_3_1_4_3}
\end{figure}

An example of an HFPL is given in Figure~\ref{Fig:Hexagonal_FPL_3_1_4_3}. In Section~\ref{Sec:Recovering_HFPLs_from_oriented_HFPLs}, local configurations around each vertex of an HFPL are considered.
It then will be necessary that each vertex of an HFPL is of degree 2. To achieve that, external edges along each boundary of an HFPL are attached as follows: given an HFPL $f$ to each vertex in 
$\mathcal{B}\cup\mathcal{T}$ a vertical external edge is attached, to each vertex in $\mathcal{L}_B\cup\mathcal{L}_T\cup\mathcal{R}_B\cup\mathcal{R}_T$ of degree $1$ a horizontal external edges is attached
and to each vertex in $\mathcal{L}_B\cup\mathcal{L}_T\cup\mathcal{R}_B\cup\mathcal{R}_T$ of degree $0$ both a horizontal and a vertical edge are attached. The so-obtained HFPL together with external edges
is denoted by $\overline{f}$. In the figures, the external edges are represented by dotted lines. \\

By the first three conditions of Definition~\ref{Def:HFPL}, an HFPL is made up of a number of paths where non-closed
paths have their extremities in \hbox{$\mathcal{L}_T\cup\mathcal{T}\cup\mathcal{R}_T\cup\mathcal{R}_B\cup\mathcal{B}\cup\mathcal{L}_B$}. 
In the following, HFPLs will be considered according to certain boundary conditions that depend on the extremities of the non-closed paths. In the particular case of a path, that has one of its extremities in
$\mathcal{T}$ (respectively in $\mathcal{B}$), roughly spoken, the boundary conditions encode whether that extremity is connected to an extremity to its left or below (respectively to the left) or wheter
it is connected to an extremity to its right (respectively to its right or above).

\begin{Def}\label{Def:BoundaryHFPL}
To each HFPL $f$ of size $(K,L,M,N)$ is assigned a sextuple of words $(\mathsf{l}_\mathsf{T}, \mathsf{t}, \mathsf{r}_\mathsf{T}; \mathsf{r}_\mathsf{B},\mathsf{b},\mathsf{l}_\mathsf{B})$ of length 
$(K,L,M;N,K+L-N,M+N-K)$ respectively in the following way:
\begin{itemize}
  \item[($\mathsf{l}_\mathsf{T}$)] If the vertex $L_{T,i}\in\mathcal{L}_T$ is of degree $1$, set $(\mathsf{l}_\mathsf{T})_i=1$, otherwise, set $(\mathsf{l}_\mathsf{T})_i=0$.
  \item[($\mathsf{t}$)\,] If the vertex $T_i\in\mathcal{T}$ is connected by a path with either a vertex in $\mathcal{L}_T\cup\mathcal{L}_B\cup\mathcal{B}$ or a vertex $T_h$ in $\mathcal{T}$, such that $h<i$, set 
	$\mathsf{t}_i=0$, otherwise, set $\mathsf{t}_i=1$.
  \item[($\mathsf{r}_\mathsf{T}$)] If the vertex $R_{T,i}\in\mathcal{R}_T$ is of degree $1$, set $(\mathsf{r}_\mathsf{T})_i=0$, otherwise, set $(\mathsf{r}_\mathsf{T})_i=1$.
  \item[($\mathsf{r}_\mathsf{B}$)] If the vertex $R_{B,i}\in\mathcal{R}_B$ is of degree $1$, set $(\mathsf{r}_\mathsf{B})_i=1$, otherwise, set $(\mathsf{r}_\mathsf{B})_i=0$.
  \item[($\mathsf{b}$)\,] If the vertex $B_i\in\mathcal{B}$ is connected by a path with either a vertex in $\mathcal{R}_T\cup\mathcal{R}_B\cup\mathcal{T}$ or a vertex $B_j$ in $\mathcal{B}$, such that $j>i$,
        set $\mathsf{b}_i=0$, otherwise, set $\mathsf{b}_i=1$.   
  \item[($\mathsf{l}_\mathsf{B}$)] If the vertex $L_{B,i}\in\mathcal{L}_{B}$ is of degree $1$, set $(\mathsf{l}_\mathsf{B})_i=0$, otherwise, set $(\mathsf{l}_\mathsf{B})_i=1$.
\end{itemize}
The HFPL $f$ is said to have \textnormal{boundary} $(\mathsf{l}_\mathsf{T},\mathsf{t},\mathsf{r}_\mathsf{T};\mathsf{r}_\mathsf{B},\mathsf{b},\mathsf{l}_\mathsf{B})$. Furthermore, 
the set of HFPLs with boundary $(\mathsf{l}_\mathsf{T},\mathsf{t},\mathsf{r}_\mathsf{T};\mathsf{r}_\mathsf{B},\mathsf{b},\mathsf{l}_\mathsf{B})$ is denoted by $H_{\mathsf{l}_\mathsf{T},\mathsf{t},\mathsf{r}_\mathsf{T}}^{\mathsf{r}_\mathsf{B},\mathsf{b},\mathsf{l}_\mathsf{B}}$ and its cardinality by $h_{\mathsf{l}_\mathsf{T},\mathsf{t},\mathsf{r}_\mathsf{T}}^{\mathsf{r}_\mathsf{B},\mathsf{b},\mathsf{l}_\mathsf{B}}$.
\end{Def}

The HFPL depicted in Figure~\ref{Fig:Hexagonal_FPL_3_1_4_3} has boundary $(011,1,0111;110,1,1110)$. With each HFPL a pair of extended link patterns will be associated, that encodes 
which pairs of extremities of non-closed paths are connected by a path.
To be more precise, to an HFPL of size $(K,L,M,N)$ 
a pair of extended link patterns 
$(\pi_b,\pi_t)$, where $\pi_b$ is an extended link pattern on $\{1,\dots,K+L-N\}$ and $\pi_t$ one on $\{1,\dots,L\}$, is assigned as follows:

\begin{enumerate}
 \item[$(\pi_b)$] In the case when $B_i, B_j\in\mathcal{B}$ are linked by a path in $f$ then $\{i,j\}\in\pi_b$. Otherwise, if $B_i\in\mathcal{B}$ is connected 
 with a vertex in $\mathcal{L}_B\cup\mathcal{L}_T$, $i$ is a left point of $\pi_b$ and, if $B_i$ is connected with a vertex in $\mathcal{T}\cup\mathcal{R}_T\cup\mathcal{R}_B$,
$i$ is a right point of $\pi_b$.
 \item[$(\pi_t)$] In the case when $T_i, T_j\in\mathcal{T}$ are linked by a 
path in $f$ then $\{i,j\}\in\pi_t$. Otherwise, 
if $T_i\in\mathcal{T}$ is connected with a vertex in 
$\mathcal{L}_B\cup\mathcal{L}_T\cup\mathcal{B}$,
$i$ is a left point of $\pi_t$ and, if $T_i$ is connected with a vertex in 
$\mathcal{R}_T\cup\mathcal{R}_B$,
$i$ is a right point of $\pi_t$.
\end{enumerate}
In Figure~\ref{Fig:HFPL_Example_5_4_5_0}, an example of an HFPL and its associated pair of extended link patterns is given. 
For any HFPL in $H_{\mathsf{l}_\mathsf{T},\mathsf{t},\mathsf{r}_\mathsf{T}}^{\mathsf{r}_\mathsf{B},\mathsf{b},\mathsf{l}_\mathsf{B}} $ with extended link patterns $\pi_b$ and $\pi_t$, it holds 
\textnormal{\textbf{w}}$(\pi_b)=\mathsf{b}$ and \textnormal{\textbf{w}}$(\pi_t)=\overline{\mathsf{t}}$. 

\begin{figure}[tbh]
\includegraphics[width=1\textwidth]{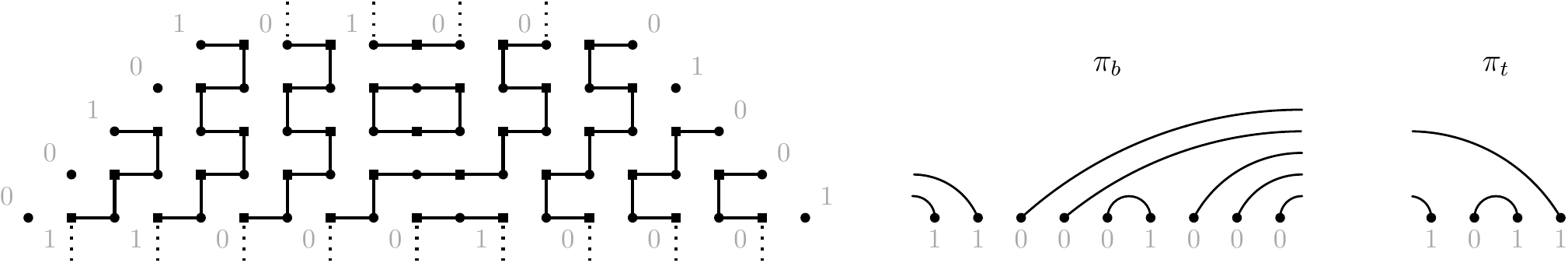}
\caption{An HFPL of size $(5,4,5,0)$ with boundary $(00101,0100,01001;\varepsilon,110001000,\varepsilon)$ and its associated pair of extended link patterns.}
\label{Fig:HFPL_Example_5_4_5_0}
\end{figure}

\subsection{Oriented hexagonal fully packed loop configurations}
The definition of HFPLs contains global conditions, as do the definitions of the top and the bottom boundary word associated with an HFPL. These global conditions can be omitted when 
adding an orientation to each edge of an HFPL.

\begin{Def}
An \textnormal{oriented HFPL} of size $(K,L,M,N)$ is an HFPL of the same size together with an orientation of each edge such that
each vertex of degree $2$ is incident to an incoming and an outgoing edge, each edge attached to a vertex in $\mathcal{L}_T\cup\mathcal{L}_B$ is outgoing and
each edge attached to a vertex in $\mathcal{R}_T\cup\mathcal{R}_B$ is incoming.
\end{Def}

In Figure~\ref{Fig:Oriented_HFPL_4_3_2_3}, an example of an oriented HFPL is given.
Since in Section~\ref{Sec:Recovering_HFPLs_from_oriented_HFPLs} it becomes necessary that each vertex in an oriented HFPL is of degree $2$,
an oriented HFPL with directed external edges $\overline{f}$ is associated with an oriented HFPL $f$ as follows: 
first, unoriented external edges are attached to $f$ in the same way as they are attached to ordinary HFPLs. Then an orientation is added to each external edge such that 
for a vertex in $\mathcal{L}_T\cup\mathcal{L}_B$ (respectively in $\mathcal{R}_T\cup\mathcal{R}_B$), that is incident to two external edges, the horizontal edge is incoming (respectively outgoing) 
and the vertical edge is outgoing (respectively incoming) and all other vertices are incident to an incoming and an outgoing edge. In the figures, the directed external edges are represented by dotted arrows.

\begin{figure}[tbh]
\includegraphics[width=.5\textwidth]{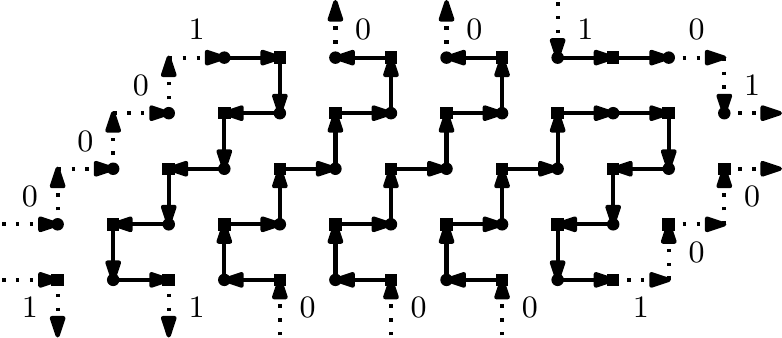}
\caption{An oriented HFPL of size $(4,3,2,3)$.}
\label{Fig:Oriented_HFPL_4_3_2_3}
\end{figure}

\begin{Def}
To each oriented HFPL its \textnormal{boundary} $(\mathsf{l}_\mathsf{T},\mathsf{t},\mathsf{r}_\mathsf{T};\mathsf{r}_\mathsf{B},\mathsf{b},\mathsf{l}_\mathsf{B})$ is assigned as follows:
\begin{enumerate}
  \item[($\mathsf{l}_\mathsf{T}$)] If $L_{T,i}\in\mathcal{L}_T$ has out-degree 1, then $(\mathsf{l}_\mathsf{T})_i=1$, otherwise, $(\mathsf{l}_\mathsf{T})_i=0$.
  \item[($\mathsf{t}$)\,] If $T_i\in\mathcal{T}$ has in-degree 1, then $\mathsf{t}_i=0$, otherwise, $\mathsf{t}_i=1$.
  \item[($\mathsf{r}_\mathsf{T}$)] If $R_{T,i}\in\mathcal{R}_T$ has in-degree 1, then $(\mathsf{r}_\mathsf{T})_i=0$, otherwise, $(\mathsf{r}_\mathsf{T})_i=1$.
  \item[($\mathsf{r}_\mathsf{B}$)] If $R_{B,i}\in\mathcal{R}_B$ has in-degree 1, then $(\mathsf{r}_\mathsf{B})_i=1$, otherwise, $(\mathsf{r}_\mathsf{B})_i=0$.
  \item[($\mathsf{b}$)\,] If $B_i\in\mathcal{B}$ has in-degree 1, then $\mathsf{b}_i=1$, otherwise, $\mathsf{b}_i=0$.
  \item[($\mathsf{l}_\mathsf{B}$)] If $L_{B,i}\in\mathcal{L}_B$ has out-degree 1, then $(\mathsf{l}_\mathsf{B})_i=0$, otherwise, $(\mathsf{l}_\mathsf{B})_i=1$.
 \end{enumerate}
The set of oriented HFPLs with boundary $(\mathsf{l}_\mathsf{T},\mathsf{t},\mathsf{r}_\mathsf{T};\mathsf{r}_\mathsf{B},\mathsf{b},\mathsf{l}_\mathsf{B})$ is denoted by $\overrightarrow{H}_{\mathsf{l}_\mathsf{T},\mathsf{t},\mathsf{r}_\mathsf{T}}^{\mathsf{r}_\mathsf{B},\mathsf{b},\mathsf{l}_\mathsf{B}}$ and its cardinality by 
$\overrightarrow{h}_{\mathsf{l}_\mathsf{T},\mathsf{t},\mathsf{r}_\mathsf{T}}^{\mathsf{r}_\mathsf{B},\mathsf{b},\mathsf{l}_\mathsf{B}}$.
\end{Def}

The oriented HFPL depicted in Figure~\ref{Fig:Oriented_HFPL_4_3_2_3} has boundary $(0001,001,01;100,1000,1)$. For oriented HFPLs, nice symmetries hold:

\begin{Prop}
\begin{enumerate}
\item Vertical reflection together with the reorientation of all edges exchanges $\overrightarrow{H}_{\mathsf{l}_\mathsf{T},\mathsf{t},\mathsf{r}_\mathsf{T}}^{\mathsf{r}_\mathsf{B},\mathsf{b},\mathsf{l}_\mathsf{B}}$ and 
$\overrightarrow{H}_{\mathsf{r}_\mathsf{T}^\ast,\mathsf{t}^\ast,\mathsf{l}_\mathsf{T}^\ast}^{\mathsf{l}_\mathsf{B}^\ast,\mathsf{b}^\ast,\mathsf{r}_\mathsf{B}^\ast}$. Thus, 
$\overrightarrow{h}_{\mathsf{l}_\mathsf{T},\mathsf{t},\mathsf{r}_\mathsf{T}}^{\mathsf{r}_\mathsf{B},\mathsf{b},\mathsf{l}_\mathsf{B}}=\overrightarrow{h}_{\mathsf{r}_\mathsf{T}^\ast,\mathsf{t}^\ast,\mathsf{l}_\mathsf{T}^\ast}^{\mathsf{l}_\mathsf{B}^\ast,\mathsf{b}^\ast,\mathsf{r}_\mathsf{B}^\ast}$.

\item Horizontal reflection exchanges $\overrightarrow{H}_{\mathsf{l}_\mathsf{T},\mathsf{t},\mathsf{r}_\mathsf{T}}^{\mathsf{r}_\mathsf{B},\mathsf{b},\mathsf{l}_\mathsf{B}}$ and
$\overrightarrow{H}_{\overline{\mathsf{l}_\mathsf{B}},\overline{\mathsf{b}},\overline{\mathsf{r}_\mathsf{B}}}^{\overline{\mathsf{r}_\mathsf{T}},\overline{\mathsf{t}},\overline{\mathsf{l}_\mathsf{T}}}$. Thus,
$\overrightarrow{h}_{\mathsf{l}_\mathsf{T},\mathsf{t},\mathsf{r}_\mathsf{T}}^{\mathsf{r}_\mathsf{B},\mathsf{b},\mathsf{l}_\mathsf{B}}=
\overrightarrow{h}_{\overline{\mathsf{l}_\mathsf{B}},\overline{\mathsf{b}},\overline{\mathsf{r}_\mathsf{B}}}^{\overline{\mathsf{r}_\mathsf{T}},\overline{\mathsf{t}},\overline{\mathsf{l}_\mathsf{T}}}$.
\end{enumerate}
\end{Prop}

There are certain constraints on the boundary of oriented HFPLs. The theorem below is the first main result of this article:

\begin{Thm}\label{Thm:NecCondHFPL}
Let $(\mathsf{l}_\mathsf{T},\mathsf{t},\mathsf{r}_\mathsf{T};\mathsf{r}_\mathsf{B},\mathsf{b},\mathsf{l}_\mathsf{B})$ be a sextuple of words of length $(K,L,M;N,K+L-N,M+N-K)$ respectively. Then
$\overrightarrow{h}_{\mathsf{l}_\mathsf{T},\mathsf{t},\mathsf{r}_\mathsf{T}}^{\mathsf{r}_\mathsf{B},\mathsf{b},\mathsf{l}_\mathsf{B}}>0$ implies:
\begin{enumerate}
 \item $\vert\mathsf{l}_\mathsf{T}\vert_0+\vert\mathsf{t}\vert_0=\vert\mathsf{r}_\mathsf{B}\vert_0+\vert\mathsf{b}\vert_0$ and $\vert\mathsf{t}\vert_1+\vert\mathsf{r}_\mathsf{T}\vert_1=\vert\mathsf{b}\vert_1+\vert\mathsf{l}_\mathsf{B}\vert_1$;
 \item $\mathsf{l}_\mathsf{T}\,\mathsf{t}\leq\mathsf{b}\,\mathsf{r}_\mathsf{B}$ and $\mathsf{t}\,\mathsf{r}_\mathsf{T}\leq\mathsf{l}_\mathsf{B}\,\mathsf{b}$ for the concatenations $\mathsf{l}_\mathsf{T}\,\mathsf{t}$,
 $\mathsf{b}\,\mathsf{r}_\mathsf{B}$, $\mathsf{t}\,\mathsf{r}_\mathsf{T}$ and $\mathsf{l}_\mathsf{B}\,\mathsf{b}$;
 \item $d(\mathsf{r}_\mathsf{B})+d(\mathsf{b})+d(\mathsf{l}_\mathsf{B})\geq d(\mathsf{l}_\mathsf{T})+d(\mathsf{t})+d(\mathsf{r}_\mathsf{T})+\vert\mathsf{l}_\mathsf{T}\vert_1\vert\mathsf{t}\vert_0+\vert\mathsf{t}\vert_1\vert\mathsf{r}_\mathsf{T}\vert_0+\vert\mathsf{r}_\mathsf{B}\vert_0\vert\mathsf{l}_\mathsf{B}\vert_1$.
\end{enumerate}
\end{Thm}

The first statement of condition (1) is equivalent to $\vert\mathsf{l}_\mathsf{T}\vert_1+\vert\mathsf{t}\vert_1=\vert\mathsf{r}_\mathsf{B}\vert_1+\vert\mathsf{b}\vert_1$ because the concatenations
$\mathsf{l}_\mathsf{T}\,\mathsf{t}$ and $\mathsf{b}\,\mathsf{r}_\mathsf{B}$ both are of length $K+L$.
On the other hand, the second statement is equivalent to $\vert\mathsf{t}\vert_0+\vert\mathsf{r}_\mathsf{T}\vert_0=\vert\mathsf{b}\vert_0+\vert\mathsf{l}_\mathsf{B}\vert_0$ since the concatenations 
$\mathsf{t}\,\mathsf{r}_\mathsf{T}$ and $\mathsf{l}_\mathsf{B}\,\mathsf{b}$ both
are of length $L+M$. In Section~\ref{Sec:PathTanglesHFPL}, a proof of Theorem~\ref{Thm:NecCondHFPL} using a model bijective to oriented HFPLs is given. It is done in an analogous way as for oriented TFPLs 
in \cite{TFPL}. \\

There is a natural injection from $H_{\mathsf{l}_\mathsf{T},\mathsf{t},\mathsf{r}_\mathsf{T}}^{\mathsf{r}_\mathsf{B},\mathsf{b},\mathsf{l}_\mathsf{B}}$ to 
$\overrightarrow{H}_{\mathsf{l}_\mathsf{T},\mathsf{t},\mathsf{r}_\mathsf{T}}^{\mathsf{r}_\mathsf{B},\mathsf{b},\mathsf{l}_\mathsf{B}}$: given an HFPL in 
$H_{\mathsf{l}_\mathsf{T},\mathsf{t},\mathsf{r}_\mathsf{T}}^{\mathsf{r}_\mathsf{B},\mathsf{b},\mathsf{l}_\mathsf{B}}$,
orient all its closed paths clockwise, each path connecting two vertices $B_i,B_j$ in $\mathcal{B}$ from $B_i$ to $B_j$, if $i<j$,
each path connecting two vertices $T_i,T_j$ in $\mathcal{T}$ from $T_i$ to $T_j$, if $i<j$, and each path connecting a vertex $B_i$ in $\mathcal{B}$ and a vertex $T_j$ in
$\mathcal{T}$ from $B_i$ to $T_j$. The other paths have a forced orientation by the definition of oriented HFPLs. Note that the chosen orientation ensures that $\mathsf{b}$ is indeed the bottom boundary word of the
resulting oriented HFPL, respectively $\mathsf{t}$ the top boundary word. Thus, this is an injection from $H_{\mathsf{l}_\mathsf{T},\mathsf{t},\mathsf{r}_\mathsf{T}}^{\mathsf{r}_\mathsf{B},\mathsf{b},\mathsf{l}_\mathsf{B}}$ to 
$\overrightarrow{H}_{\mathsf{l}_\mathsf{T},\mathsf{t},\mathsf{r}_\mathsf{T}}^{\mathsf{r}_\mathsf{B},\mathsf{b},\mathsf{l}_\mathsf{B}}$ and therefore it holds

\begin{equation}
h_{\mathsf{l}_\mathsf{T},\mathsf{t},\mathsf{r}_\mathsf{T}}^{\mathsf{r}_\mathsf{B},\mathsf{b},\mathsf{l}_\mathsf{B}}\leq\overrightarrow{h}_{\mathsf{l}_\mathsf{T},\mathsf{t},\mathsf{r}_\mathsf{T}}^{\mathsf{r}_\mathsf{B},\mathsf{b},\mathsf{l}_\mathsf{B}}
\label{Eq:HFPLcontainedinOHFPL}
\end{equation}
for any $\mathsf{l}_\mathsf{T},\mathsf{t},\mathsf{r}_\mathsf{T},\mathsf{r}_\mathsf{B},\mathsf{b},\mathsf{l}_\mathsf{B}$. In the other direction, with each oriented HFPL a non-oriented HFPL can be associated by 
ignoring the orientation of the edges. This operation does not preserve
the bottom and the top words in general. In Section~\ref{Sec:Recovering_HFPLs_from_oriented_HFPLs}, it is shown how to deduce the number 
$h_{\mathsf{l}_\mathsf{T},\mathsf{t},\mathsf{r}_\mathsf{T}}^{\mathsf{r}_\mathsf{B},\mathsf{b},\mathsf{l}_\mathsf{B}}$ from a certain 
weighted enumeration of oriented HFPLs.
From (\ref{Eq:HFPLcontainedinOHFPL}) the following corollary of Theorem~\ref{Thm:NecCondHFPL} is obtained immediately:

\begin{Cor}
The conclusions of Theorem~\ref{Thm:NecCondHFPL} hold if $h_{\mathsf{l}_\mathsf{T},\mathsf{t},\mathsf{r}_\mathsf{T}}^{\mathsf{r}_\mathsf{B},\mathsf{b},\mathsf{l}_\mathsf{B}}>0$.
\end{Cor}
 
To each oriented HFPL a pair $(\overrightarrow{\pi_b},\overrightarrow{\pi_t})$ of directed extended link patterns is assigned in the natural way. 
\begin{figure}[tbh]
\includegraphics[width=1\textwidth]{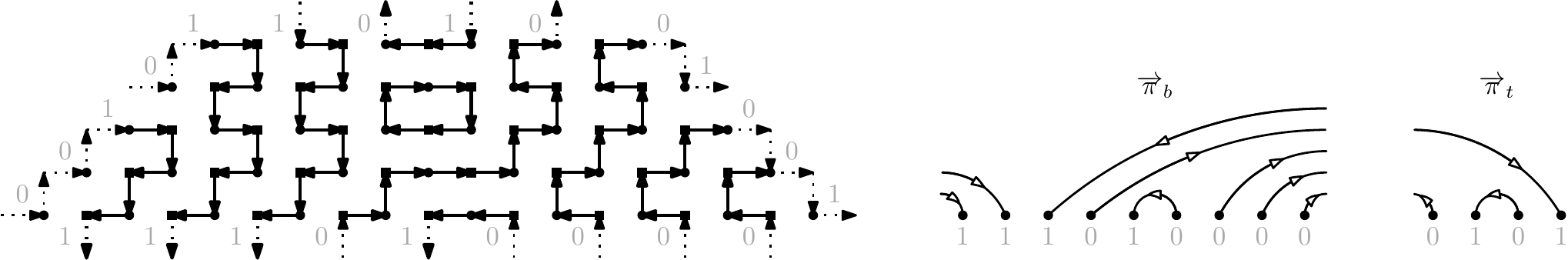}
\caption{An oriented HFPL of size $(5,4,5,0)$ together with its two corresponding extended link patterns.}
\label{Fig:Oriented_HFPL_Example_5_4_5_0}
\end{figure}
In Figure~\ref{Fig:Oriented_HFPL_Example_5_4_5_0}, an example of an oriented HFPL and its assigned pair of directed extended link patterns is given. 
Note that $\overrightarrow{\pi_b}$ has to be left-hand-incoming and $\overrightarrow{\pi_t}$ has to be right-hand-outgoing. Furthermore, 
for any oriented HFPL in $\overrightarrow{H}_{\mathsf{l}_\mathsf{T},\mathsf{t},\mathsf{r}_\mathsf{T}}^{\mathsf{r}_\mathsf{B},\mathsf{b},\mathsf{l}_\mathsf{B}}$ the source-sink-word of $\overrightarrow{\pi_b}$ equals 
$\mathsf{b}$ and the source-sink-word of $\overrightarrow{\pi_t}$ equals $\overline{\mathsf{t}}$.\\
 
\section{Recovering HPFLs from oriented HFPLs}\label{Sec:Recovering_HFPLs_from_oriented_HFPLs}
In this section the interplay between HFPLs and oriented HFPLs is studied. It is done analogous to the study of the interplay between TFPLs and
oriented TFPLs in \cite{TFPL}. For that reason, only a rough overview will be given.

\subsection{The weighted enumeration of oriented HFPLs} 
A step is said to be of type \textbf{u} if it is a 
$(0,1)$-step, of type \textbf{l} if it is a $(-1,0)$-step and of type \textbf{d} if it is a $(0,-1)$-step. Furthermore, a turn is said to be of type \textbf{dl} if it consists of a step of type \textbf{d} that is 
suceeded by a step of type \textbf{l}, of type \textbf{lu} if it consists of a step of type \textbf{l} that is suceeded by a step of type \textbf{u}, of type \textbf{ld} if it consists of a step of type \textbf{l} that 
is suceeded by a step of type \textbf{d} and of type \textbf{ul} if it consist of a step of type \textbf{u} that is suceeded by a step of type \textbf{l}. In the following, set 
$R=\{\textbf{dl},\textbf{lu}\}$ and $L=\{\textbf{ld},\textbf{ul}\}$.\\

\begin{figure}[tbh]
\includegraphics[width=.15\textwidth]{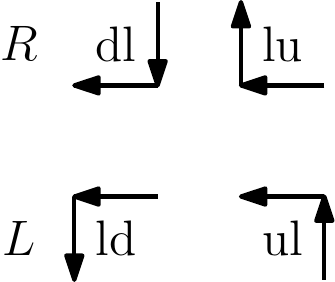}
\caption{The four types of turns.}
\end{figure}

In the following, fix a turn $t_{\circlearrowleft}\in L$ and let $t_{\circlearrowright}$ be the turn in $R$ that is obtained by swapping the two steps in $t_{\circlearrowleft}$. Furthermore,
the number of occurrences of turns of type $t_{\circlearrowleft}$ (respectively of type $t_{\circlearrowright}$) in $\overline{f}$ 
is denoted by $t_{\circlearrowleft}(\overline{f})$ (respectively $t_{\circlearrowright}(\overline{f})$), where $f$ is an oriented HFPL.
The difference $t_{\circlearrowleft}(\overline{f})-t_{\circlearrowright}(\overline{f})$ has the following global interpretation:

\begin{Prop}\label{Prop:InterpretationTurnsHFPLs}
Let $f$ be an oriented HFPL and $\overrightarrow{\pi_b}$ and $\overrightarrow{\pi_t}$ be the two directed extended link patterns that are associated with $f$.
Furthermore, set $RL_b(f)=RL(\overrightarrow{\pi_b})$, $RL_t(f)=RL(\overrightarrow{\pi_t})$ and denote by $N^{\circlearrowright}(f)$ (respectively $N^{\circlearrowleft}(f)$) the number of closed paths in $f$ that are oriented clockwise
(respectively counter-clockwise). Then
\begin{equation}
t_{\circlearrowleft}(\overline{f})-t_{\circlearrowright}(\overline{f})=RL_b(f)-RL_t(f)+N^{\circlearrowleft}(f)-N^{\circlearrowright}(f).
\notag\end{equation}
\end{Prop}

\begin{proof}
The proof of Proposition~\ref{Prop:InterpretationTurnsHFPLs} is a generalization of the proof of Proposition~2.4 in \cite{TFPL}. Essential for the proof is the following assertion that is given in \cite[Corollary~2.3]{TFPL}:
\textit{for all directed closed self-avoiding paths $p$ on the square lattice, $t_{\circlearrowleft}(p)-t_{\circlearrowright}(p)$ equals -1 (resp. 1) if $p$ is oriented clockwise (resp. counter-clockwise).}
Here, $t_{\circlearrowleft}(p)$ denotes the number of occurrences of turns of type $t_{\circlearrowleft}$ in $p$, respectively $t_{\circlearrowright}(p)$ the number of occurrences of turns of type
$t_{\circlearrowright}$ in $p$. It remains to evaluate $t_{\circlearrowright}(p)-t_{\circlearrowleft}(p)$ for the non-closed paths $p$ in $\overline{f}$. 
In the following, the external edges are considered part of the non-closed paths.

As a start, let $p$ be a non-closed path in $\overline{f}$ that connects two vertices in $\mathcal{T}$, see Figure~\ref{Fig:Closure_Path_HFPL} in a particular case. 
Then $p$ starts with a step of type \textbf{d} and ends with a step of type \textbf{u}. 
Now, $p$ is completed to a closed self-avoiding path $p'$ on the square lattice by adding a path above the configuration with the least possible number of turns. 
If $p$ goes from $T_j$ to $T_i$ with $i<j$, then $p'$ is oriented clockwise and it follows that $-1=t_{\circlearrowleft}(p)-t_{\circlearrowright}(p)$. 
On the other hand, if $p$ goes from $T_i$ to $T_j$ with $i<j$, 
then $p'$ is oriented counter-clockwise and it follows that $t_{\circlearrowleft}(p)-t_{\circlearrowright}(p)=0$. 

\begin{figure}[tbh]
\includegraphics[width=.4\textwidth]{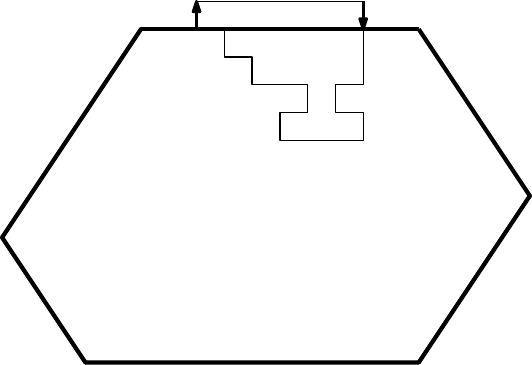}
\caption{Closure of a path in an oriented HFPL.}
\label{Fig:Closure_Path_HFPL}
\end{figure}

Next, let $p$ be a non-closed path in $\overline{f}$ that connects a vertex in $\mathcal{B}$ and a vertex in $\mathcal{T}$. In that case, $p$ starts with a step of type \textbf{u} and ends with a step of type \textbf{d}
or vice versa. The non-closed path $p$ is completed to a closed self-avoiding path $p'$ on the square lattice by adding a path to the right of $\overline{f}$ with the least possible number of turns. 
If $p$ is oriented from the vertex in $\mathcal{B}$ to the vertex in $\mathcal{T}$, then $p'$ is oriented clockwise and therefore 
$t_{\circlearrowleft}(p)-t_{\circlearrowright}(p)=0$. On the other hand if $p$ is oriented from the vertex in $\mathcal{T}$ to the 
vertex in $\mathcal{B}$ then $p'$ is oriented counter-clockwise and one obtains again $t_{\circlearrowleft}(p)-t_{\circlearrowright}(p)=0$.

Next, let $p$ be a non-closed path in $\overline{f}$ that goes from a vertex in $\mathcal{T}$ to a vertex in $\mathcal{R}_T\cup\mathcal{R}_B$. 
In that case $p$ starts with a step of type \textbf{d} and ends with a step
of type \textbf{r}. Now, $p$ is completed to a closed self-avoiding path $p'$ by adding a path above $\overline{f}$ with the least possible number of turns.
Then $p'$ is oriented counter-clockwise and therefore 
$t_{\circlearrowleft}(p)-t_{\circlearrowright}(p)=0$. The difference also vanishes if $p$ goes from a vertex in $\mathcal{T}$ to a 
vertex in $\mathcal{L}_B\cup\mathcal{L}_T$.\\

Finally, $t_{\circlearrowright}(p)-t_{\circlearrowleft}(p)=1$ if $p$ goes from a vertex $B_j$ and $B_i$ with $i<j$ and $t_{\circlearrowright}(p)-t_{\circlearrowleft}(p)=0$ if $p$ goes from a vertex $B_i$ to a vertex 
$B_j$ with $i<j$ or to
a vertex in $\mathcal{L}_B\cup \mathcal{L}_T\cup\mathcal{R}_B\cup\mathcal{R}_T$ 
or if $p$ goes from a vertex in $\mathcal{L}_B\cup \mathcal{L}_T$ to a vertex in $\mathcal{R}_B\cup\mathcal{R}_T$ 
by Proposition~2.4 in \cite{TFPL}.
\end{proof}
In particular, Proposition~\ref{Prop:InterpretationTurnsHFPLs}
implies that the numbers $t_{\circlearrowleft}(\overline{f})-t_{\circlearrowright}(\overline{f})$ for any oriented HFPL $f$ do not depend on the choice of $t_{\circlearrowleft}\in L$.

\begin{Prop}\label{Prop:Weights_HFPL}
Let $f$ be an oriented HFPL. Then 
\begin{equation}
\vcenter{\hbox{\includegraphics[width=.05\textwidth]{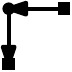}}}+\vcenter{\hbox{\includegraphics[width=.05\textwidth]{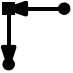}}}+
-\vcenter{\hbox{\includegraphics[width=.05\textwidth]{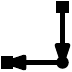}}}-\vcenter{\hbox{\includegraphics[width=.05\textwidth]{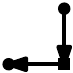}}}=
\vcenter{\hbox{\includegraphics[width=.05\textwidth]{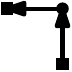}}}+ \vcenter{\hbox{\includegraphics[width=.05\textwidth]{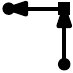}}}
-\vcenter{\hbox{\includegraphics[width=.05\textwidth]{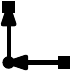}}}-\vcenter{\hbox{\includegraphics[width=.05\textwidth]{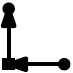}}}.
\notag\end{equation}
Here, $\vcenter{\hbox{\includegraphics[width=.05\textwidth]{Int_Pairs_3}}}$, etc. denote the number of occurrences of the local configurations $\vcenter{\hbox{\includegraphics[width=.05\textwidth]{Int_Pairs_3}}}$, etc.
\end{Prop}

This motivates the following weighted enumeration of oriented HFPLs: let
$t_{\circlearrowleft}\in L$ and $t_{\circlearrowright}\in R$ be the turn, that is obtained by swapping the two steps in $t_{\circlearrowleft}$. Define

\begin{equation}
\overrightarrow{h}_{\mathsf{l}_\mathsf{T},\mathsf{t},\mathsf{r}_\mathsf{T}}^{\mathsf{r}_\mathsf{B},\mathsf{b},\mathsf{l}_\mathsf{B}}(q)=\sum\limits_{f\in\overrightarrow{H}_{\mathsf{l}_\mathsf{T},\mathsf{t},\mathsf{r}_\mathsf{T}}^{\mathsf{r}_\mathsf{B},\mathsf{b},\mathsf{l}_\mathsf{B}}}q^{t_{\circlearrowleft}(\overline{f})-t_{\circlearrowright}(\overline{f})}=
\sum\limits_{f\in\overrightarrow{H}_{\mathsf{l}_\mathsf{T},\mathsf{t},\mathsf{r}_\mathsf{T}}^{\mathsf{r}_\mathsf{B},\mathsf{b},\mathsf{l}_\mathsf{B}}}q^{RL_b(f)-RL_t(f)}q^{N^{\circlearrowleft}(f)-N^{\circlearrowright}(f)}.
\label{Eq:WeightedEnumerationHFPLs} 
\end{equation}

\subsection{Deriving the number of ordinary HFPLs from the weighted enumeration of oriented HFPLs} The goal of this subsection is to extract the number of HFPLs with boundary 
$(\mathsf{l}_\mathsf{T},\mathsf{t},\mathsf{r}_\mathsf{T};\mathsf{r}_\mathsf{B},\mathsf{b},\mathsf{l}_\mathsf{B})$
from the weighted enumeration of 
oriented HFPLs in (\ref{Eq:WeightedEnumerationHFPLs}). For that purpose, let $\overline{H}_{\mathsf{l}_\mathsf{T},\mathsf{t},\mathsf{r}_\mathsf{T}}^{\mathsf{r}_\mathsf{B},\mathsf{b},\mathsf{l}_\mathsf{B}}$
denote the subset of $\overrightarrow{H}_{\mathsf{l}_\mathsf{T},\mathsf{t},\mathsf{r}_\mathsf{T}}^{\mathsf{r}_\mathsf{B},\mathsf{b},\mathsf{l}_\mathsf{B}}$ that is made up of those oriented TFPLs 
whose associated directed link patterns $\overrightarrow{\pi_b}$ and $\overrightarrow{\pi_t}$ verify $RL(\overrightarrow{\pi_b})=0$ and $RL(\overrightarrow{\pi_t})=0$. 
Furthermore, let $\overline{h}_{\mathsf{l}_\mathsf{T},\mathsf{t},\mathsf{r}_\mathsf{T}}^{\mathsf{r}_\mathsf{B},\mathsf{b},\mathsf{l}_\mathsf{B}}(q)$
be the corresponding weighted enumeration, cf. (\ref{Eq:WeightedEnumerationHFPLs}). The following lemma relates $h_{\mathsf{l}_\mathsf{T},\mathsf{t},\mathsf{r}_\mathsf{T}}^{\mathsf{r}_\mathsf{B},\mathsf{b},\mathsf{l}_\mathsf{B}}$
to $\overline{h}_{\mathsf{l}_\mathsf{T},\mathsf{t},\mathsf{r}_\mathsf{T}}^{\mathsf{r}_\mathsf{B},\mathsf{b},\mathsf{l}_\mathsf{B}}(q)$:

\begin{Lemma}\label{Lemma:HFPL_Weighted_Enumeration_Without_RL}
Let $\rho$ be a primitive sixth root of unity, so that $\rho$ satisfies $\rho+1/\rho=1$. Then 
\begin{equation}
 \overline{h}_{\mathsf{l}_\mathsf{T},\mathsf{t},\mathsf{r}_\mathsf{T}}^{\mathsf{r}_\mathsf{B},\mathsf{b},\mathsf{l}_\mathsf{B}}(\rho)=h_{\mathsf{l}_\mathsf{T},\mathsf{t},\mathsf{r}_\mathsf{T}}^{\mathsf{r}_\mathsf{B},\mathsf{b},\mathsf{l}_\mathsf{B}}.
\notag\end{equation}
\end{Lemma}

The arguments in the proof of Lemma~\ref{Lemma:HFPL_Weighted_Enumeration_Without_RL} are the same as in the proof of an analogous identity for TFPLs in \cite[Proposition 2.5]{TFPL}. For that reason, the proof is omitted. 
Given an oriented HFPL $f$ in $\overrightarrow{H}_{\mathsf{l}_\mathsf{T},\mathsf{t},\mathsf{r}_\mathsf{T}}^{\mathsf{r}_\mathsf{B},\mathsf{b},\mathsf{l}_\mathsf{B}}$, consider the oriented HFPL, 
that is obtained from $f$ by orienting all paths in $f$ that connect two vertices in $\mathcal{B}$ or two vertices in $\mathcal{T}$ from left to right. Its boundary has to be
$(\mathsf{l}_\mathsf{T},\mathsf{t}',\mathsf{r}_\mathsf{T};\mathsf{r}_\mathsf{B},\mathsf{b}',\mathsf{l}_\mathsf{B})$ for a word $\mathsf{b}'$ that is left-points-fixing feasible for $\mathsf{b}$ and a word
$\mathsf{t}'$ that is right-points-fixing feasible for $\mathsf{t}$. Furthermore, its weight is decreased by $q^{g(\mathsf{b},\mathsf{b}')-g(\mathsf{t},\mathsf{t}')}$.
For those reasons, the following holds: 

\begin{equation} 
\overrightarrow{h}_{\mathsf{l}_\mathsf{T},\mathsf{t},\mathsf{r}_\mathsf{T}}^{\mathsf{r}_\mathsf{B},\mathsf{b},\mathsf{l}_\mathsf{B}}(q)=
\sum\limits_{\substack{\mathsf{t}':\,\mathsf{t}' \textnormal{ right-points-fixing feasible for } \mathsf{t}\\
\mathsf{b}': \,\mathsf{b}' \textnormal{ left-points-fixing feasible for } \mathsf{b}}}q^{-g(\mathsf{t},\mathsf{t}')}q^{g(\mathsf{b},\mathsf{b}')}
\overline{h}_{\mathsf{l}_\mathsf{T},\mathsf{t}^{\prime},\mathsf{r}_\mathsf{T}}^{\mathsf{r}_\mathsf{B},\mathsf{b}^{\prime},\mathsf{l}_\mathsf{B}}(q)
\label{Eq:Weighted_Relations_HFPL}\end{equation}

The goal is to invert the relation in (\ref{Eq:Weighted_Relations_HFPL}) so that with the help of Lemma~\ref{Lemma:HFPL_Weighted_Enumeration_Without_RL} an expression of the number of HFPLs in terms of the weighted
enumeration of oriented HFPLs in (\ref{Eq:WeightedEnumerationHFPLs}) is gained.

\begin{Def}
\begin{enumerate}
\item The square matrix $M=M_b(n)$ of size $n$ has rows and columns indexed by words of length $n$ and entry $M_{w,w'}=q^{g(w,w')}$, if $w'$ is left-points-fixing feasible for $w$, and entry $M_{w,w'}=0$, otherwise.
\item The square matrix $M=M_t(n)$ of size $n$ has rows and columns indexed by words of length $n$ and entry $M_{w,w'}=q^{-g(w,w')}$, if $w'$ is right-points-fixing feasible for $w$, and entry $M_{w,w'}=0$ otherwise.
\end{enumerate} 
\end{Def}
These are square matrices of size $2^n$.
\begin{figure}[tbh]
\centering
\includegraphics[width=1\textwidth]{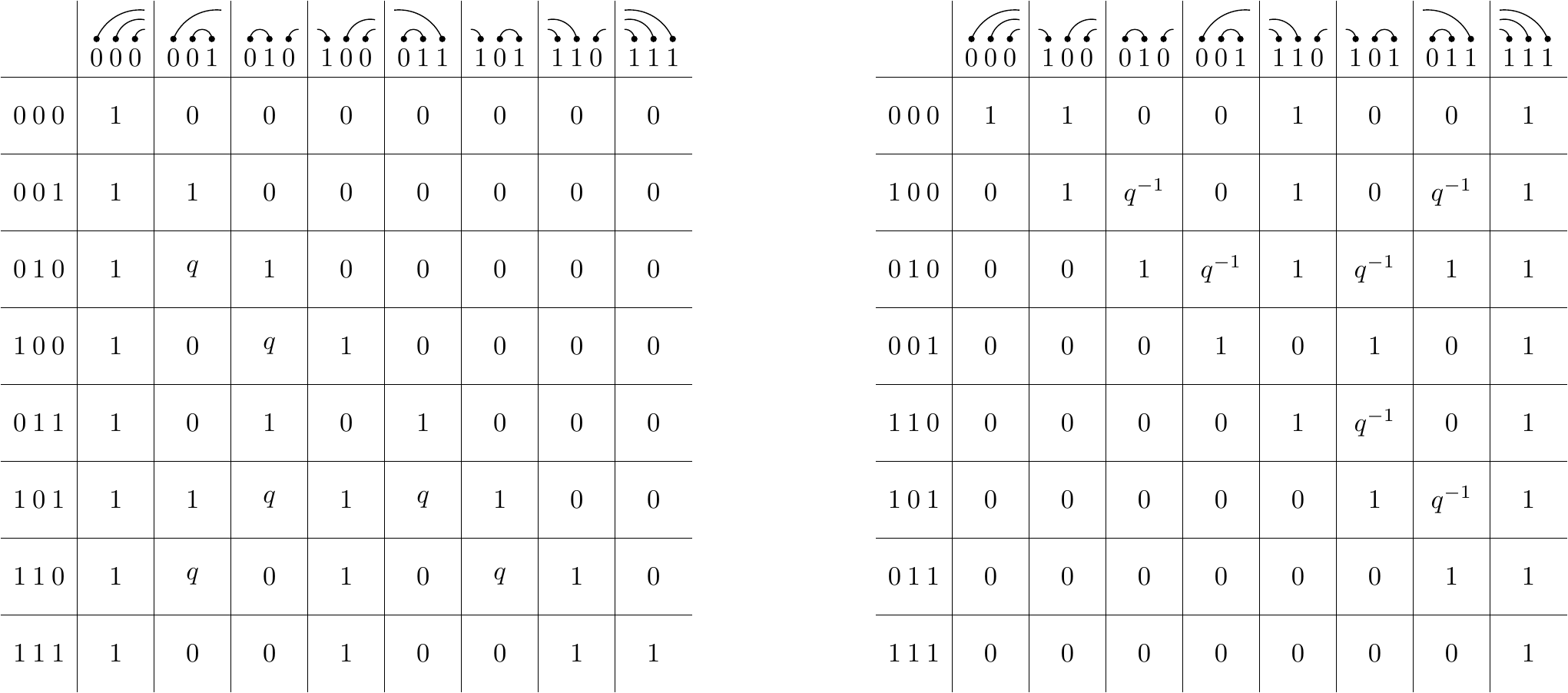}
\caption{The matrices $M_b(3)$ (left) and $M_t(3)$ (right).}
\end{figure}

\begin{Prop}\label{Prop:HFPL_Matrices_invertible}
For any positive integer $n$ the matrices $M_b(n)$ and $M_t(n)$ are invertible.  
\end{Prop}

\begin{proof}
Throughout this proof, if $w'$ is feasible for $w$ let $\overrightarrow{\pi}$ be the unique directed extended link pattern with underlying extended link pattern $\textbf{w}^{-1}(w')$ and source-sink word $w$.

It will first be proven that $M_b(n)$ is a lower triangular matrix with ones on the diagonal and therefore invertible. There are only ones on the diagonal of $M_b(n)$ because $q^{g(w,w)}=q^0=1$ for all words $w$ of 
length $n$.
To show that $M_b(n)$ is lower triangular it is sufficient to find a linear order $\prec_b$ on the set of words of length $n$ that satisfies $w'\prec_b w$ whenever $w'\neq w$ and $w'$ is left-points-fixing feasible for $w$ and use it
for the rows and columns of $M_b(n)$.
First, note that if $w'$ is left-points-fixing feasible for $w$ and in $\overrightarrow{\pi}$ all right points are sinks then 
there exist ordered pairs $(i_1,j_1), (i_2,j_2), \dots (i_k,j_k)$ such that $w'_{i_s}=0$, $w'_{j_s}=1$, $w_{i_s}=1$ and $w_{j_s}=0$ for all $1\leq s\leq k$ and $w_i=w'_i$ for all other indices. 
Thus, $\vert w'\vert_1=\vert w\vert_1$ and $w'\leq w$ in that particular case. Now, given any two words $w$ and $w'$ of length $n$ such that
$w'$ is left-points-fixing feasible for $w$ then $\vert w\vert_1-\vert w'\vert_1$ is the number of right points in $\overrightarrow{\pi}$
which are sinks. In particular, $\vert w\vert_1\geq\vert w'\vert_1$ in that case. 
Hence, for to given words $w$ and $w'$ set $w'\preceq_b w$ if $\vert w'\vert_1\leq \vert w\vert_1$ and in the case when $\vert w'\vert_1=\vert w\vert_1$ if additionally
$w'\leq w$. Then, by $\preceq_b$ a partial order on the set of words of length $n$ is defined. Furthermore, for any two words $w$ and $w'$ of length $n$ such that $w'$ is left-points-fixing feasibility for $w$ it follows 
$w'\preceq_b w$.
Thus, for any linear order $\prec_b$ on the set of words of length $n$ that extends $\preceq_b$ it holds that $w'\prec_b w$ whenever $w'\neq w$ and $w'$ is left-points-fixing feasible for $w$. 

Finally, it will be shown that $M_t(n)$ is an upper triangular matrix with ones on the diagonal and therefore invertible. Different to above, a linear order $\prec_t$ on the set of words of length $n$ that satisfies
$w' \succ_t w$ whenever $w'\neq w$ and $w'$ is right-points-fixing feasible for $w$ is needed to be found. In the case when
$w'$ is right-points-fixing feasible for $w$, the number of left points that are sources is given by $\vert w'\vert_1-\vert w\vert_1$. Thus, $\vert w'\vert_1\geq \vert w\vert_1$ in that case. 
Now, a partial order $\preceq_t$ on the set of words of length $n$ is defined as follows: given two words $w$ and $w'$ of length $n$ it is set $w'\preceq_t w$ if $\vert w\vert_1\leq 
\vert w'\vert_1$ and in the case when $\vert w\vert_1=\vert w'\vert_1$ if additionally $w\leq w'$. 
Then right-points-fixing
feasibility of $w'$ for $w$ implies that $w'\succeq_t w$
Summing up, given any linear order on the set of words of length $n$ that extends $\preceq_t$ and using that order for the 
rows and columns of $M_t(n)$, the matrix $M_t(n)$ becomes an upper triangular matrix with ones on the diagonal. 
\end{proof}

\begin{Cor}\label{Cor:WeightedEnumeration_HFPL}
Let $(\mathsf{l}_\mathsf{T},\mathsf{t},\mathsf{r}_\mathsf{T};\mathsf{r}_\mathsf{B},\mathsf{b},\mathsf{l}_\mathsf{B})$ be a sextuple of words of length $(K,L,M;N,K+L-N,M+N-K)$ respectively. Then

\begin{equation}
\overline{h}_{\mathsf{l}_\mathsf{T},\mathsf{t},\mathsf{r}_\mathsf{T}}^{\mathsf{r}_\mathsf{B},\mathsf{b},\mathsf{l}_\mathsf{B}}(q)=
\sum\limits_{\mathsf{t}', \mathsf{b}'}(M_b(K+L-N)^{-1})_{\mathsf{b},\mathsf{b}'}(M_t(L)^{-1})_{\mathsf{t},\mathsf{t}'}
\overrightarrow{h}_{\mathsf{l}_\mathsf{T},\mathsf{t}',\mathsf{r}_\mathsf{T}}^{\mathsf{r}_\mathsf{B},\mathsf{b}',\mathsf{l}_\mathsf{B}}(q) 
\notag\end{equation}
and in particular

\begin{equation}
h_{\mathsf{l}_\mathsf{T},\mathsf{t},\mathsf{r}_\mathsf{T}}^{\mathsf{r}_\mathsf{B},\mathsf{b},\mathsf{l}_\mathsf{B}}=\sum\limits_{\mathsf{t}', \mathsf{b}'}
(M_b(K+L-N)^{-1})_{\mathsf{b},\mathsf{b}'}(M_t(L)^{-1})_{\mathsf{t},\mathsf{t}'}\overrightarrow{h}_{\mathsf{l}_\mathsf{T},\mathsf{t}',\mathsf{r}_\mathsf{T}}^{\mathsf{r}_\mathsf{B},\mathsf{b}',\mathsf{l}_\mathsf{B}}(\rho)
\notag\end{equation}
where $\rho$ is a primitive sixth root of unity.
\end{Cor}

\section{Path-tangles}\label{Sec:PathTanglesHFPL}

In the first part of this section, new objects which will turn out to be in bijection with oriented HFPLs are etablished: hexagonal blue-red path-tangles. They are essential for the proof of Theorem~\ref{Thm:NecCondHFPL}.
In the second part, a purely combinatorial proof of Theorem~\ref{Thm:NecCondHFPL}(3) is given.
The idea of the proof is the same as in the proof of Theorem~4.3 in \cite{TFPL}. 

\subsection{Path-tangles}\label{Subsec:PathTanglesHFPL}
Throughout this subsection, when not mentioned otherwise, $\mathsf{l}_\mathsf{T}$, $\mathsf{t}$, $\mathsf{r}_\mathsf{T}$, $\mathsf{r}_\mathsf{B}$, $\mathsf{b}$ and $\mathsf{l}_\mathsf{B}$ are considered with the 
additional constraints
$\vert\mathsf{l}_\mathsf{T}\vert_0+\vert\mathsf{t}\vert_0=\vert\mathsf{b}\vert_0+\vert\mathsf{r}_\mathsf{B}\vert_0$ and $\vert\mathsf{t}\vert_1+\vert\mathsf{r}_\mathsf{T}\vert_1=\vert\mathsf{l}_\mathsf{B}\vert_1+\vert
\mathsf{b}\vert_1$.

\begin{figure}[tbh]
\includegraphics[width=.75\textwidth]{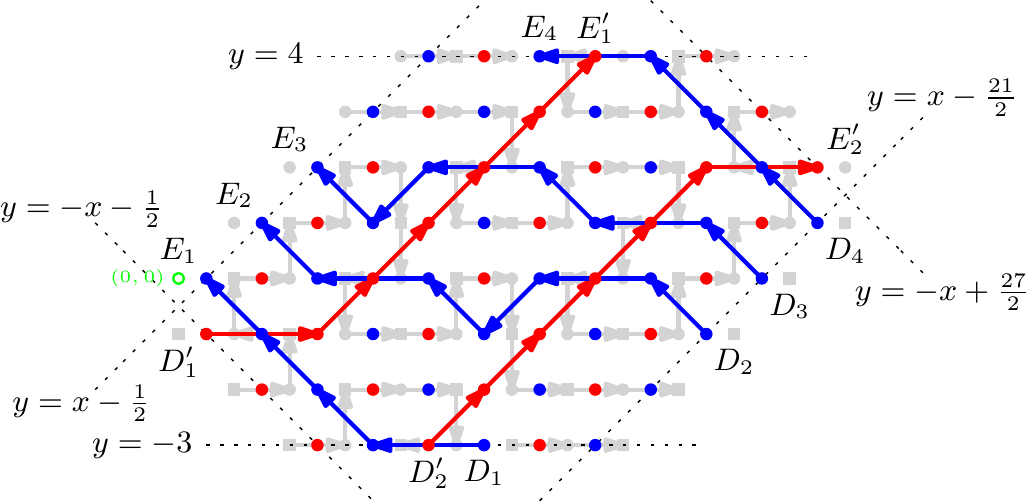} 
\caption{A hexagonal blue-red path-tangle with boundary $(00011,01,001;11000,10,100)$. The oriented HFPL it corresponds to is indicated in gray.}
\label{Fig:Hexagonal_path-tangle_5_2_3_5}
\end{figure}

In the following, let $I_{\mathsf{b}\,\mathsf{r}_\mathsf{B}}=\{1\leq i_1<\cdots < i_{\vert\mathsf{b}\vert_0+\vert\mathsf{r}_\mathsf{B}\vert_0}\leq K+L\}$ be the set of indices $i$ such that 
$(\mathsf{b}\,\mathsf{r}_\mathsf{B})_i=0$ and 
$I_{\mathsf{l}_\mathsf{T}\,\mathsf{t}}=\{1\leq j_1<\cdots< j_{\vert\mathsf{l}_\mathsf{T}\vert_0+\vert\mathsf{t}\vert_0}\leq K+L\}$ the set of indices $j$ such that $(\mathsf{l}_\mathsf{T}\,\mathsf{t})_j=0$. 
Futhermore, set
\begin{equation}
D_k=\begin{cases} (M+N-K-\frac{3}{2}+2i_k,-M-N+K) & \textnormal{ for } k=1,2,\dots,\vert\mathsf{b}\vert_0,\\
		  (M+L-\frac{1}{2}+i_k,-M-L-1+i_k) & \textnormal{ for } k=\vert\mathsf{b}\vert_0+1, \vert\mathsf{b}\vert_0+2,\dots,\vert\mathsf{b}\vert_0+\vert\mathsf{r}_\mathsf{B}\vert_0
    \end{cases} 
\notag\end{equation}
and
\begin{equation}
E_\ell=\begin{cases}  (j_\ell-\frac{1}{2},j_\ell-1) & \textnormal{ for } \ell=1,2,\dots,\vert\mathsf{l}_T\vert_0,\\
        (2j_\ell-K-\frac{1}{2},K-1) & \textnormal{ for } \ell=\vert\mathsf{l}_T\vert_0+1,\vert\mathsf{l}_T\vert_0+2,\dots,\vert\mathsf{l}_T\vert_0+\vert\mathsf{t}\vert_0.
       \end{cases} 
\notag\end{equation}
Let $\mathcal{P}(D_k,E_\ell)$ denote the set of paths from $D_k$ to $E_\ell$ using steps $(-1,1)$, $(-1,-1)$ and $(-2,0)$ which
never go below the line $y=-M-N+K$ and never above the line $y=K-1$.\\

On the other hand, let $I'_{\mathsf{l}_\mathsf{B}\,\mathsf{b}}=\{1\leq i'_1<\cdots < i'_{\vert\mathsf{l}_\mathsf{B}\vert_1+\vert\mathsf{b}\vert_1}\leq L+M\}$ be the set of indices $i'$ such that 
$(\mathsf{l}_\mathsf{B}\,\mathsf{b})_{i'}=1$ and 
$I'_{\mathsf{t}\,\mathsf{r}_\mathsf{T}}=\{1\leq j'_1<\cdots< j'_{\vert\mathsf{t}\vert_1+\vert\mathsf{r}_\mathsf{T}\vert_1}\leq L+M\}$ the set of indices $j'$ such that $(\mathsf{t}\,\mathsf{r}_\mathsf{T})_{j'}=1$. 
Furthermore, set
\begin{equation} 
D'_k=
 \begin{cases} (i'_k-\frac{1}{2},-i'_k) & \textnormal{ for } k=1,2,\dots,\vert\mathsf{l}_\mathsf{B}\vert_1\\
(2i'_k-M-N+K-\frac{1}{2},-M-N+K) & \textnormal{ for } k=\vert\mathsf{l}_\mathsf{B}\vert_1+1,\vert\mathsf{l}_\mathsf{B}\vert_1+2,\dots,\vert\mathsf{l}_\mathsf{B}\vert_1+\vert\mathsf{b}\vert_1
 \end{cases}
\notag\end{equation}
and
\begin{equation} 
E'_\ell=
 \begin{cases}
  (K+2j'_\ell-\frac{3}{2},K-1) & \textnormal{ for } \ell=1,2,\dots,\vert\mathsf{t}\vert_1,\\
  (K+L+j'_\ell-\frac{1}{2},K+L-j'_\ell) & \textnormal{ for } \ell=\vert\mathsf{t}\vert_1+1,\vert\mathsf{t}\vert_1+2,\dots,\vert\mathsf{t}\vert_1+\vert\mathsf{r}_\mathsf{T}\vert_1.
 \end{cases}
\notag\end{equation}
Let $\mathcal{P}'(D'_k,E'_\ell)$ denote the set of paths from $D'_k$ to $E'_\ell$ using steps $(1,1)$, $(1,-1)$ and $(2,0)$ which
never go below the line $y=-M-N+K$ and above the line $y=K-1$.

\begin{Def}
Let $\mathcal{P}(\mathsf{b}\,\mathsf{r}_\mathsf{B},\mathsf{l}_\mathsf{T}\,\mathsf{t})$ denote the set of $(\vert\mathsf{l}_\mathsf{T}\vert_0+\vert\mathsf{t}\vert_0)$-tuples $(P_1,\dots,P_{\vert\mathsf{l}_\mathsf{T}\vert_0+\vert\mathsf{t}\vert_0})$ of non-intersecting paths $P_k\in\mathcal{P}(D_k,E_k)$
and $\mathcal{P}'(\mathsf{l}_\mathsf{B}\,\mathsf{b},\mathsf{t}\,\mathsf{r}_\mathsf{T})$ denote the set of $(\vert\mathsf{t}\vert_1+\vert\mathsf{r}_\mathsf{T}\vert_1)$-tuples $(P'_1,\dots,P'_{\vert\mathsf{t}\vert_1+\vert\mathsf{r}_\mathsf{T}\vert_1})$ of non-intersecting paths 
$P'_k\in\mathcal{P}'(D'_k,E'_k)$. Furthermore, let \textnormal{BlueRed}$(\mathsf{l}_\mathsf{T},\mathsf{t},\mathsf{r}_\mathsf{T};\mathsf{r}_\mathsf{B},\mathsf{b},\mathsf{l}_\mathsf{B})$ denote the set of pairs $(B,R)\in\mathcal{P}(\mathsf{b}\mathsf{r}_\mathsf{B},\mathsf{l}_\mathsf{T}\mathsf{t})\times\mathcal{P}'(\mathsf{l}_\mathsf{B}\mathsf{b},\mathsf{t}\mathsf{r}_\mathsf{T})$
that satisfy the following:

\begin{enumerate}
 \item No diagonal step of $R$ crosses a diagonal step of $B$.
 \item Each middle point of a horizontal step of $B$ (resp. of $R$) is used by a step in $R$ (resp. $B$).  
\end{enumerate}
A configuration in \textnormal{BlueRed}$(\mathsf{l}_\mathsf{T},\mathsf{t},\mathsf{r}_\mathsf{T};\mathsf{r}_\mathsf{B},\mathsf{b},\mathsf{l}_\mathsf{B})$ is said to be a 
\textnormal{hexagonal blue-red path-tangle} with boundary $(\mathsf{l}_\mathsf{T},\mathsf{t},\mathsf{r}_\mathsf{T};\mathsf{r}_\mathsf{B},\mathsf{b},\mathsf{l}_\mathsf{B})$.
\end{Def}

In Figure~\ref{Fig:Hexagonal_path-tangle_5_2_3_5}, an example of a blue-red path-tangle with boundary $(00011,01,001;11000,10,100)$ is given.
Hexagonal blue-red path-tangles with boundary $(\mathsf{l}_\mathsf{T},\mathsf{t},\mathsf{r}_\mathsf{T};\mathsf{r}_\mathsf{B},\mathsf{b},\mathsf{l}_\mathsf{B})$ encode oriented HFPLs
with boundary $(\mathsf{l}_\mathsf{T},\mathsf{t},\mathsf{r}_\mathsf{T};\mathsf{r}_\mathsf{B},\mathsf{b},\mathsf{l}_\mathsf{B})$: given an oriented HFPL 
$f\in\overrightarrow{H}_{\mathsf{l}_\mathsf{T},\mathsf{t},\mathsf{r}_\mathsf{T}}^{\mathsf{r}_\mathsf{B},\mathsf{b},\mathsf{l}_\mathsf{B}}$, \textit{blue vertices} are added to $f$ in the middle of 
each horizontal line of $H^{K,L,M,N}$ having an odd left and an even right vertex and \textit{red vertices} are added in the middle of each horizontal line of $H^{K,L,M,N}$ having an even left and an odd right vertex.
Then blue and red arrows are added as indicated in Figure~\ref{Fig:Bijection:HFPL_Pathtangles}.
\begin{figure}[tbh]
\includegraphics[width=1\textwidth]{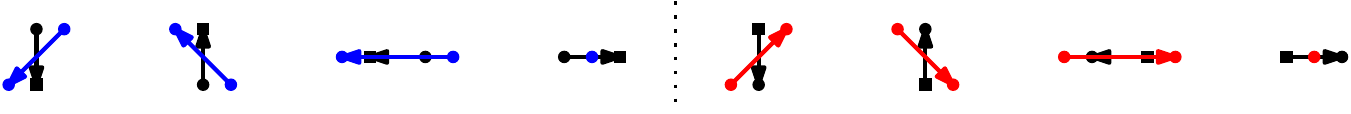}
\caption{From oriented HFPLs to blue-red path-tangles.}
\label{Fig:Bijection:HFPL_Pathtangles}
\end{figure}
After removing all vertices and edges of $f$ a blue-red path-tangle in BlueRed$(\mathsf{l}_\mathsf{T},\mathsf{t},\mathsf{r}_\mathsf{T};\mathsf{r}_\mathsf{B},\mathsf{b},\mathsf{l}_\mathsf{B})$ is obtained.

\begin{Thm}\label{Thm:Bijection_HFPL_path-tangle}
The map described above is a bijection between $\overrightarrow{H}_{\mathsf{l}_\mathsf{T},\mathsf{t},\mathsf{r}_\mathsf{T}}^{\mathsf{r}_\mathsf{B},\mathsf{b},\mathsf{l}_\mathsf{B}}$ and 
BlueRed$(\mathsf{l}_\mathsf{T},\mathsf{t},\mathsf{r}_\mathsf{T};\mathsf{r}_\mathsf{B},\mathsf{b},\mathsf{l}_\mathsf{B})$.
\end{Thm}

In Figure~\ref{Fig:Hexagonal_path-tangle_5_2_3_5}, the oriented HFPL corresponding to the depicted blue-red path-tangle is indicated in the same figure. The proof of Theorem~\ref{Thm:Bijection_HFPL_path-tangle}
is omitted because the arguments of the proof of Theorem~4.1 in \cite{TFPL} also apply for oriented HFPLs respectively hexagonal blue-red path-tangles. 
An immediate consequence of Theorem~\ref{Thm:Bijection_HFPL_path-tangle} is that the boundary $(\mathsf{l}_\mathsf{T},\mathsf{t},\mathsf{r}_\mathsf{T};\mathsf{r}_\mathsf{B},\mathsf{b},\mathsf{l}_\mathsf{B})$ of an 
oriented HFPL has to satisfy 
$\vert\mathsf{l}_\mathsf{T}\vert_0+\vert\mathsf{t}\vert_0=\vert\mathsf{b}\vert_0+\vert\mathsf{r}_\mathsf{B}\vert_0$ and 
$\vert\mathsf{t}\vert_1+\vert\mathsf{r}_\mathsf{T}\vert_1=\vert\mathsf{l}_\mathsf{B}\vert_1+\vert\mathsf{b}\vert_1$, what are the assertions of Theorem~\ref{Thm:NecCondHFPL}(1). 
Also the constraints on the boundary of an oriented HFPL stated in Theorem~\ref{Thm:NecCondHFPL}(2) can now be proven.

\begin{proof}[Proof of Theorem~\ref{Thm:NecCondHFPL}(2)]
It will only be shown that $\mathsf{l}_\mathsf{T}\,\mathsf{t}\leq\mathsf{b}\,\mathsf{r}_\mathsf{B}$.
Let $(B,R)\in\textnormal{BlueRed}(\mathsf{l}_\mathsf{T},\mathsf{t},\mathsf{r}_\mathsf{T};\mathsf{r}_\mathsf{B},\mathsf{r}_\mathsf{T},\mathsf{l}_\mathsf{B})$ where
$B=(P_1,\dots,P_{\vert\mathsf{l}_\mathsf{T}\vert_0+\vert\mathsf{t}\vert_0})$ for blue non-intersecting paths $P_k\in\mathcal{P}(D_k,E_k)$. 
Each path $P_k$ only uses steps $(-1,-1)$, $(-1,1)$ and $(-2,0)$. For that reason, the number of $(-1,-1)$- and of $(-2,0)$-steps of a path $P_k$ is given by $i_k-j_k$, where $i_k$ and $j_k$ are as defined above.
In particular, $j_k\leq i_k$ for each $1\leq k\leq \vert\mathsf{l}_{\mathsf{T}}\vert_0+\vert\mathsf{t}\vert_0$. Therefore, $\mathsf{l}_{\mathsf{T}}\,\mathsf{t}\leq \mathsf{b}\,\mathsf{r}_{\mathsf{B}}$.
\end{proof}

\begin{Prop}\label{Prop:HFPL_Enumeration_Edges} For any oriented HFPL in $\overrightarrow{H}_{\mathsf{l}_\mathsf{T},\mathsf{t},\mathsf{r}_\mathsf{T}}^{\mathsf{r}_\mathsf{B},\mathsf{b},\mathsf{l}_\mathsf{B}}$
and for any path-tangle in 
BlueRed$(\mathsf{l}_\mathsf{T},\mathsf{t},\mathsf{r}_\mathsf{T};\mathsf{r}_\mathsf{B},\mathsf{r}_\mathsf{T},\mathsf{l}_\mathsf{B})$ respectively, the following two formulas hold:

\begin{enumerate}
 \item $d(\mathsf{r}_\mathsf{B})+d(\mathsf{b})+\vert\mathsf{r}_\mathsf{B}\vert_0\vert\mathsf{b}\vert_1-d(\mathsf{l}_\mathsf{T})-d(\mathsf{t})-\vert\mathsf{l}_\mathsf{T}\vert_1\vert\mathsf{t}\vert_0
 =\vcenter{\hbox{\includegraphics[width=.05\textwidth]{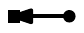}}}+\vcenter{\hbox{\includegraphics[width=.012\textwidth]{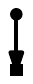}}}
 =\vcenter{\hbox{\includegraphics[width=.09\textwidth]{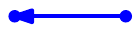}}}+\vcenter{\hbox{\includegraphics[width=.05\textwidth]{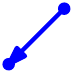}}}$;
 \item $d(\mathsf{b})+d(\mathsf{l}_\mathsf{B})+\vert\mathsf{b}\vert_0\vert\mathsf{l}_\mathsf{B}\vert_1-d(\mathsf{t})-d(\mathsf{r}_\mathsf{T})-\vert\mathsf{t}\vert_1\vert\mathsf{r}_\mathsf{T}\vert_0
 =\vcenter{\hbox{\includegraphics[width=.05\textwidth]{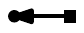}}}
 +\vcenter{\hbox{\includegraphics[width=.012\textwidth]{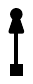}}}
 =\vcenter{\hbox{\includegraphics[width=.09\textwidth]{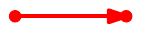}}}
 +\vcenter{\hbox{\includegraphics[width=.05\textwidth]{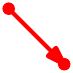}}}$.
 \end{enumerate}
Here, $\vcenter{\hbox{\includegraphics[width=.05\textwidth]{blue_left_HFPL}}}$, etc. denote the numbers of occurrences of the edge $\vcenter{\hbox{\includegraphics[width=.05\textwidth]{blue_left_HFPL}}}$, etc.
\end{Prop}

\begin{proof} As already noted in the proof of Thereom~\ref{Thm:NecCondHFPL}(2) above, it holds 

\begin{equation}
\sum\limits_{k=1}^{\vert\mathsf{l}_\mathsf{T}\vert_0+\vert\mathsf{t}\vert_0}(i_k-j_k)=
\vcenter{\hbox{\includegraphics[width=.09\textwidth]{blue_left}}}+\vcenter{\hbox{\includegraphics[width=.05\textwidth]{blue_down}}}.
\notag\end{equation}
On the other hand, 
\begin{align}
i_k-j_k & =\# \textnormal{ of }1's \textnormal{ among the first } i_k \textnormal{ letters of } \mathsf{b}\,\mathsf{r}_\mathsf{B} -\#\textnormal{ of }1's \textnormal{ among the first } j_k 
\textnormal{ letters of } \mathsf{l}_\mathsf{T}\,\mathsf{t}.
\notag\end{align}
Thus, 
\begin{align}
\sum\limits_{k=1}^{\vert\mathsf{l}_\mathsf{T}\vert_0+\vert\mathsf{t}\vert_0}(i_k-j_k)=d(\mathsf{b}\mathsf{r}_\mathsf{B})-d(\mathsf{l}_\mathsf{T}\mathsf{t})= 
d(\mathsf{r}_\mathsf{B})+d(\mathsf{b})+\vert\mathsf{r}_\mathsf{B}\vert_0\vert\mathsf{b}\vert_1-d(\mathsf{l}_\mathsf{T})-d(\mathsf{t})-\vert\mathsf{l}_\mathsf{T}\vert_1\vert\mathsf{t}\vert_0,
\notag\end{align}
what proves the first identity. The second identity follows analogously.
\end{proof}

\subsection{The interpretation of $d(\mathsf{r}_\mathsf{B})+d(\mathsf{b})+d(\mathsf{l}_\mathsf{B})-d(\mathsf{l}_\mathsf{T})-d(\mathsf{t})-d(\mathsf{r}_\mathsf{T})-
\vert\mathsf{l}_\mathsf{T}\vert_1\vert\mathsf{t}\vert_0-\vert\mathsf{t}\vert_1\vert\mathsf{r}_\mathsf{T}\vert_0-\vert\mathsf{r}_\mathsf{B}\vert_0\vert\mathsf{l}_\mathsf{B}\vert_1$}\label{Sec:Combinatorial_Interpretation_HFPL}
In this subsection, it will be shown that given an oriented HFPL with boundary $(\mathsf{l}_\mathsf{T},\mathsf{t},\mathsf{r}_\mathsf{T};\mathsf{r}_\mathsf{B},\mathsf{b},\mathsf{l}_\mathsf{B})$ by the quantity
$d(\mathsf{r}_\mathsf{B})+d(\mathsf{b})+d(\mathsf{l}_\mathsf{B})-d(\mathsf{l}_\mathsf{T})-d(\mathsf{t})-d(\mathsf{r}_\mathsf{T})-\vert\mathsf{l}_\mathsf{T}\vert_1\vert\mathsf{t}\vert_0-\vert\mathsf{t}\vert_1\vert\mathsf{r}_\mathsf{T}\vert_0-\vert\mathsf{r}_\mathsf{B}\vert_0\vert\mathsf{l}_\mathsf{B}\vert_1$
the occurrences of certain local patterns are counted. Throughout this subsection, the numbers of occurrences of the local configurations $\vcenter{\hbox{\includegraphics[width=.1\textwidth]{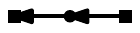}}}$, etc.
are denoted by $\vcenter{\hbox{\includegraphics[width=.1\textwidth]{Expose7}}}$, etc.

\begin{Thm}\label{Thm:HFPL_Comb_Int}
For any oriented HFPL in $\overrightarrow{H}_{\mathsf{l}_\mathsf{T},\mathsf{t},\mathsf{r}_\mathsf{T}}^{\mathsf{r}_\mathsf{B},\mathsf{b},\mathsf{l}_\mathsf{B}}$ the following formula holds:
\begin{align}
 d(\mathsf{r}_\mathsf{B})+d(\mathsf{b})+d(\mathsf{l}_\mathsf{B}) & -d(\mathsf{l}_\mathsf{T})-d(\mathsf{t})-d(\mathsf{r}_\mathsf{T})-\vert\mathsf{l}_\mathsf{T}\vert_1\vert\mathsf{t}\vert_0-\vert\mathsf{t}\vert_1\vert\mathsf{r}_\mathsf{T}\vert_0-\vert\mathsf{r}_\mathsf{B}\vert_0\vert\mathsf{l}_\mathsf{B}\vert_1\notag\\
 & = \vcenter{\hbox{\includegraphics[width=.014\textwidth]{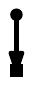}}}+ 
	      \vcenter{\hbox{\includegraphics[width=.014\textwidth]{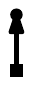}}}+\vcenter{\hbox{\includegraphics[width=.1\textwidth]{Expose7}}}+
	      \vcenter{\hbox{\includegraphics[width=.1\textwidth]{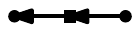}}}+\vcenter{\hbox{\includegraphics[width=.06\textwidth]{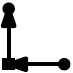}}}+
	      \vcenter{\hbox{\includegraphics[width=.06\textwidth]{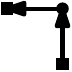}}}+\vcenter{\hbox{\includegraphics[width=.06\textwidth]{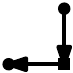}}}+
		\vcenter{\hbox{\includegraphics[width=.06\textwidth]{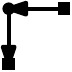}}}
\label{Eq:HFPL_Comb_Int}\end{align}
\end{Thm}

An immediate consequence of Theorem~\ref{Thm:HFPL_Comb_Int} is condition (3) of Theorem~\ref{Thm:NecCondHFPL}.
The proof of Theorem~\ref{Thm:HFPL_Comb_Int} is done in terms of blue-red path-tangles. To show (\ref{Eq:HFPL_Comb_Int}), a few identities for blue-red path-tangles are needed.

\begin{Def} In a blue-red path-tangle, a pair $(b,r)$ consisting of a blue path $b$ and a red path $r$ is said to be \textnormal{intersecting} if $b$ and $r$ intersect at least once.
\end{Def}

The number of intersecting pairs of a blue-red path-tangle in BlueRed($\mathsf{l}_\mathsf{T},\mathsf{t},\mathsf{r}_\mathsf{T};\mathsf{r}_\mathsf{B},\mathsf{r}_\mathsf{T},\mathsf{l}_\mathsf{B}$) on the one hand
can be derived from $\mathsf{l}_\mathsf{T},\mathsf{t},\mathsf{r}_\mathsf{T},\mathsf{r}_\mathsf{B},\mathsf{r}_\mathsf{T},\mathsf{l}_\mathsf{B}$ and on the other hand can be expressed in terms of numbers of occurrences
of certain local configurations.

\begin{Lemma}\label{Lem:HFPL_Intersecting_pairs} For any blue-red path-tangle in BlueRed$(\mathsf{l}_\mathsf{T},\mathsf{t},\mathsf{r}_\mathsf{T};\mathsf{r}_\mathsf{B},\mathsf{b},\mathsf{l}_\mathsf{B})$, the number of 
its intersecting pairs of paths equals
\begin{equation}
d(\mathsf{b})-d(\mathsf{t})+\vert\mathsf{b}\vert_0\vert\mathsf{l}_\mathsf{B}\vert_1+\vert\mathsf{r}_\mathsf{B}\vert_0(\vert\mathsf{b}\vert_1+\vert\mathsf{l}_\mathsf{B}\vert_1).
\label{Eq:HFPL_Intersecting_pairs}
\end{equation} 
\end{Lemma}

\begin{proof} 
Let $r$ be a red path of a blue-red path-tangle $(B,R)$ in BlueRed$(\mathsf{l}_\mathsf{T},\mathsf{t},\mathsf{r}_\mathsf{T};\mathsf{r}_\mathsf{B},\mathsf{b},\mathsf{l}_\mathsf{B})$. 
It is started with counting the blue paths that intersect with $r$. By the definition of blue-red path-tangles in Section~\ref{Subsec:PathTanglesHFPL}, there exists a 
$k\in\{1,2,\dots,\vert\mathsf{t}\vert_1+\vert\mathsf{r}_\mathsf{T}\vert_1\}$ such that $r$ has starting point $D_k'$ and ending point $E_k'$. The following three cases for $k$ are distinguished: 
$k\leq \operatorname{min}\{\vert\mathsf{t}\vert_1,\vert\mathsf{l}_\mathsf{B}\vert_1\}$, $k> \operatorname{max}\{\vert\mathsf{t}\vert_1,\vert\mathsf{l}_\mathsf{B}\vert_1\}$, 
$\operatorname{min}\{\vert\mathsf{t}\vert_1,\vert\mathsf{l}_\mathsf{B}\vert_1\}<k\leq \operatorname{max}\{\vert\mathsf{t}\vert_1,\vert\mathsf{l}_\mathsf{B}\vert_1\}$. 
In the case when $k\leq \operatorname{min}\{\vert\mathsf{t}\vert_1,\vert\mathsf{l}_\mathsf{B}\vert_1\}$ the number of blue paths, that intersect with $r$, equals 
\begin{equation}
\vert\mathsf{l}_\mathsf{T}\vert_0+\vert\mathsf{t}\vert_0-\#\textnormal{ of 0's among the last } (L-j'_k)\textnormal{ letters of } \mathsf{t}. 
\notag\end{equation}
Furthermore, in the case when $k>\textnormal{max}\{\vert\mathsf{t}\vert_1,\vert\mathsf{l}_\mathsf{B}\vert_1\}$, the number of blue paths, that intersect with $r$, equals
\begin{equation}
\vert\mathsf{r}_\mathsf{B}\vert_0+\vert\mathsf{b}\vert_0-\#\textnormal{ of 0's among the last } (L+M-j'_k) 
\textnormal{ letters of } \mathsf{b}.
\notag\end{equation}
Finally, in the case when $\operatorname{min}\{\vert\mathsf{t}\vert_1,\vert\mathsf{l}_\mathsf{B}\vert_1\}<k\leq \operatorname{max}\{\vert\mathsf{t}\vert_1,\vert\mathsf{l}_\mathsf{B}\vert_1\}$,
another destinction is necessary: whether $\vert\mathsf{t}\vert_1<\vert\mathsf{l}_\mathsf{B}\vert_1$ or whether $\vert\mathsf{l}_\mathsf{B}\vert_1<\vert\mathsf{t}\vert_1$.
If $\vert\mathsf{t}\vert_1< \vert\mathsf{l}_\mathsf{B}\vert_1$, the number of blue paths, that intersect with $r$, equals
\begin{equation}
\vert\mathsf{l}_\mathsf{T}\vert_0+\vert\mathsf{t}\vert_0.
\notag\end{equation}
On the other hand, if $\vert\mathsf{l}_\mathsf{B}\vert_1< \vert\mathsf{t}\vert_1$, the number of blue paths, that intersects with $r$, equals
\begin{equation}
 \vert\mathsf{r}_\mathsf{B}\vert_0+\#\textnormal{ of 0's among the last } (L+M-j'_k) \textnormal{ letters of }\mathsf{b} -\#\textnormal{ of 0's among the last }(L-j'_k)\textnormal{ letters of } \mathsf{t}.
\notag\end{equation}
By summing the numbers of blue paths that intersect with $r$ over all red paths $r$ of $(B,R)$, the number of intersecting pairs of $(B,R)$ is obtained and it equals the quantity in (\ref{Eq:HFPL_Intersecting_pairs}).
\end{proof}

Expressing the number of intersecting pairs of a blue-red path-tangle in terms of numbers of occurrences of certain local patterns gives, together with Lemma~\ref{Lem:HFPL_Intersecting_pairs}, the following identities:

\begin{Lemma}\label{Lem:HFPL_Intersecting_pairs_comb_int}
For any oriented HFPL and for any blue-red path-tangle with boundary $(\mathsf{l}_\mathsf{T},\mathsf{t},\mathsf{r}_\mathsf{T};\mathsf{r}_\mathsf{B},\mathsf{b},\mathsf{l}_\mathsf{B})$ respectively, one has
\begin{align}
d(\mathsf{b}) -d(\mathsf{t})+\vert\mathsf{b}\vert_0\vert\mathsf{l}_\mathsf{B}\vert_1+\vert\mathsf{r}_\mathsf{B}\vert_0(\vert\mathsf{b}\vert_1+\vert\mathsf{l}_\mathsf{B}\vert_1) & 
= \vcenter{\hbox{\includegraphics[width=.05\textwidth]{Int_Pairs_1}}}
+\vcenter{\hbox{\includegraphics[width=.05\textwidth]{Int_Pairs_2}}}-\vcenter{\hbox{\includegraphics[width=.05\textwidth]{Int_Pairs_3}}}-\vcenter{\hbox{\includegraphics[width=.05\textwidth]{Int_Pairs_4}}}\notag\\
& = \vcenter{\hbox{\includegraphics[width=.085\textwidth]{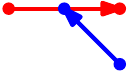}}}+\vcenter{\hbox{\includegraphics[width=.085\textwidth]{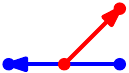}}}-\vcenter{\hbox{\includegraphics[width=.085\textwidth]{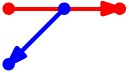}}}
-\vcenter{\hbox{\includegraphics[width=.085\textwidth]{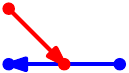}}}, \textnormal{ } and \notag\\
\notag\\
d(\mathsf{b}) -d(\mathsf{t})+\vert\mathsf{b}\vert_0\vert\mathsf{l}_\mathsf{B}\vert_1+\vert\mathsf{r}_\mathsf{B}\vert_0(\vert\mathsf{b}\vert_1+\vert\mathsf{l}_\mathsf{B}\vert_1) & = 
\vcenter{\hbox{\includegraphics[width=.05\textwidth]{Int_Pairs_9}}}
+\vcenter{\hbox{\includegraphics[width=.05\textwidth]{Int_Pairs_10}}}-\vcenter{\hbox{\includegraphics[width=.05\textwidth]{Int_Pairs_11}}}-\vcenter{\hbox{\includegraphics[width=.05\textwidth]{Int_Pairs_12}}}\notag\\
& = \vcenter{\hbox{\includegraphics[width=.085\textwidth]{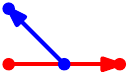}}}+\vcenter{\hbox{\includegraphics[width=.085\textwidth]{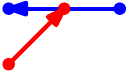}}}-\vcenter{\hbox{\includegraphics[width=.085\textwidth]{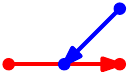}}}
-\vcenter{\hbox{\includegraphics[width=.085\textwidth]{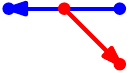}}}.
\notag\end{align}
\end{Lemma}

Lemma~\ref{Lem:HFPL_Intersecting_pairs_comb_int} generalizes Lemma~4.7 in \cite{TFPL}, which states an analogous identity for oriented TFPLs, to oriented HFPLs. 
Now, everything that is needed to prove Theorem~\ref{Thm:HFPL_Comb_Int} is provided.

\begin{proof}[Proof of Theorem~\ref{Thm:HFPL_Comb_Int}]
As a start, notice that Proposition~\ref{Prop:HFPL_Enumeration_Edges} implies the following identity:
\begin{align}
d(\mathsf{r}_\mathsf{B})+d(\mathsf{b})+d(\mathsf{l}_\mathsf{B})-d(\mathsf{l}_\mathsf{T})-d(\mathsf{t})- & d(\mathsf{r}_\mathsf{T})-\vert\mathsf{l}_\mathsf{T}\vert_1\vert\mathsf{t}\vert_0-\vert\mathsf{t}\vert_1\vert\mathsf{r}_\mathsf{T}\vert_0-\vert\mathsf{r}_\mathsf{B}\vert_0\vert\mathsf{l}_\mathsf{B}\vert_1\notag\\
& =\vcenter{\hbox{\includegraphics[width=.012\textwidth]{blue_down_HFPL}}}+\vcenter{\hbox{\includegraphics[width=.012\textwidth]{red_down_HFPL}}}+\vcenter{\hbox{\includegraphics[width=.05\textwidth]{red_right_HFPL}}}+
\vcenter{\hbox{\includegraphics[width=.05\textwidth]{blue_left_HFPL}}}-(d(\mathsf{b})-d(\mathsf{t})+\vert\mathsf{b}\vert_0\vert\mathsf{l}_\mathsf{B}\vert_1+\vert\mathsf{r}_\mathsf{B}\vert_0(\vert\mathsf{l}_\mathsf{B}\vert_1+\vert\mathsf{b}\vert_1)).
\notag\end{align}
It remains to consider the right hand side of the previous equation. 
The number $\vcenter{\hbox{\includegraphics[width=.05\textwidth]{red_right_HFPL}}}+
\vcenter{\hbox{\includegraphics[width=.05\textwidth]{blue_left_HFPL}}}$ can be expressed in the following way:
\begin{align}
\vcenter{\hbox{\includegraphics[width=.05\textwidth]{red_right_HFPL}}}+\vcenter{\hbox{\includegraphics[width=.05\textwidth]{blue_left_HFPL}}}& = 
\frac{1}{2}(\vcenter{\hbox{\includegraphics[width=.05\textwidth]{red_right_HFPL}}}+\vcenter{\hbox{\includegraphics[width=.05\textwidth]{red_right_HFPL}}})+\frac{1}{2}
(\vcenter{\hbox{\includegraphics[width=.05\textwidth]{blue_left_HFPL}}}+
\vcenter{\hbox{\includegraphics[width=.05\textwidth]{blue_left_HFPL}}})\notag\\
& = \frac{1}{2}\left(\vcenter{\hbox{\includegraphics[width=.05\textwidth]{Int_Pairs_11}}}+\vcenter{\hbox{\includegraphics[width=.1\textwidth]{Expose8}}}+\vcenter{\hbox{\includegraphics[width=.05\textwidth]{Int_Pairs_1}}}
+\vcenter{\hbox{\includegraphics[width=.05\textwidth]{Int_Pairs_9}}}+\vcenter{\hbox{\includegraphics[width=.1\textwidth]{Expose7}}}+\vcenter{\hbox{\includegraphics[width=.05\textwidth]{Int_Pairs_3}}}
\right)\notag\\
& + \frac{1}{2}\left(\vcenter{\hbox{\includegraphics[width=.05\textwidth]{Int_Pairs_2}}}+\vcenter{\hbox{\includegraphics[width=.1\textwidth]{Expose7}}}
+\vcenter{\hbox{\includegraphics[width=.05\textwidth]{Int_Pairs_12}}}+\vcenter{\hbox{\includegraphics[width=.05\textwidth]{Int_Pairs_4}}}
+\vcenter{\hbox{\includegraphics[width=.1\textwidth]{Expose8}}}+\vcenter{\hbox{\includegraphics[width=.05\textwidth]{Int_Pairs_10}}}
\right).
\label{Eq:Proof_Comb_Int_HFPL_1}\end{align}
On the other hand, by Lemma~\ref{Lem:HFPL_Intersecting_pairs_comb_int},
\begin{align}
d(\mathsf{b}) -d(\mathsf{t})+ & \vert\mathsf{b}\vert_0\vert\mathsf{l}_\mathsf{B}\vert_1+\vert\mathsf{r}_\mathsf{B}\vert_0(\vert\mathsf{b}\vert_1+\vert\mathsf{l}_\mathsf{B}\vert_1)\notag\\
& = \frac{1}{2}(d(\mathsf{b}) -d(\mathsf{t})+\vert\mathsf{b}\vert_0\vert\mathsf{l}_\mathsf{B}\vert_1+\vert\mathsf{r}_\mathsf{B}\vert_0(\vert\mathsf{b}\vert_1+\vert\mathsf{l}_\mathsf{B}\vert_1)+
d(\mathsf{b}) -d(\mathsf{t})+\vert\mathsf{b}\vert_0\vert\mathsf{l}_\mathsf{B}\vert_1+\vert\mathsf{r}_\mathsf{B}\vert_0(\vert\mathsf{b}\vert_1+\vert\mathsf{l}_\mathsf{B}\vert_1))\notag\\
& = \frac{1}{2}\left(
\vcenter{\hbox{\includegraphics[width=.05\textwidth]{Int_Pairs_1}}}+\vcenter{\hbox{\includegraphics[width=.05\textwidth]{Int_Pairs_2}}}
-\vcenter{\hbox{\includegraphics[width=.05\textwidth]{Int_Pairs_3}}}-\vcenter{\hbox{\includegraphics[width=.05\textwidth]{Int_Pairs_4}}}+
\vcenter{\hbox{\includegraphics[width=.05\textwidth]{Int_Pairs_9}}}+\vcenter{\hbox{\includegraphics[width=.05\textwidth]{Int_Pairs_10}}}-
\vcenter{\hbox{\includegraphics[width=.05\textwidth]{Int_Pairs_11}}}-\vcenter{\hbox{\includegraphics[width=.05\textwidth]{Int_Pairs_12}}}
\right).
\label{Eq:Proof_Comb_Int_HFPL_2}\end{align}
Finally, subtracting (\ref{Eq:Proof_Comb_Int_HFPL_2}) from (\ref{Eq:Proof_Comb_Int_HFPL_1}) gives the identity of Theorem~\ref{Thm:HFPL_Comb_Int}.
\end{proof}

\section{Configurations of small excess}\label{Sec:Configurations_small_excess_HFPL}

By Theorem~\ref{Thm:NecCondHFPL}(3), there is no oriented HFPL with boundary $(\mathsf{l}_\mathsf{T},\mathsf{t},\mathsf{r}_\mathsf{T};\mathsf{r}_\mathsf{B},\mathsf{b},\mathsf{l}_\mathsf{B})$ unless the integer  
$d(\mathsf{r}_\mathsf{B})+d(\mathsf{b})+d(\mathsf{l}_\mathsf{B}) -d(\mathsf{l}_\mathsf{T})-d(\mathsf{t})-d(\mathsf{r}_\mathsf{T})-\vert\mathsf{l}_\mathsf{T}\vert_1\vert\mathsf{t}\vert_0-
\vert\mathsf{t}\vert_1 \vert\mathsf{r}_\mathsf{T}\vert_0-\vert\mathsf{r}_\mathsf{B}\vert_0\vert\mathsf{l}_\mathsf{B}\vert_1$ is non-negative.

\begin{Def}
Given a sextuple $(\mathsf{l}_\mathsf{T},\mathsf{t},\mathsf{r}_\mathsf{T};\mathsf{r}_\mathsf{B},\mathsf{b},\mathsf{l}_\mathsf{B})$ of words of length $(K,L,M;N,K+L-N,M+N-K)$ respectively, 
its \textnormal{excess} is defined as
\begin{equation}
exc(\mathsf{l}_\mathsf{T},\mathsf{t},\mathsf{r}_\mathsf{T};\mathsf{r}_\mathsf{B},\mathsf{b},\mathsf{l}_\mathsf{B})= d(\mathsf{r}_\mathsf{B})+d(\mathsf{b})+d(\mathsf{l}_\mathsf{B}) -d(\mathsf{l}_\mathsf{T})-d(\mathsf{t})-d(\mathsf{r}_\mathsf{T})-\vert\mathsf{l}_\mathsf{T}\vert_1\vert\mathsf{t}\vert_0-\vert\mathsf{t}\vert_1\vert\mathsf{r}_\mathsf{T}\vert_0-\vert\mathsf{r}_\mathsf{B}\vert_0\vert\mathsf{l}_\mathsf{B}\vert_1
.\notag\end{equation}
In the case when $exc(\mathsf{l}_\mathsf{T},\mathsf{t},\mathsf{r}_\mathsf{T};\mathsf{r}_\mathsf{B},\mathsf{b},\mathsf{l}_\mathsf{B})=k$, an oriented HFPL in $\overrightarrow{H}_{\mathsf{l}_\mathsf{T},\mathsf{t},\mathsf{r}_\mathsf{T}}^{\mathsf{r}_\mathsf{B},\mathsf{b},\mathsf{l}_\mathsf{B}}$ 
is said to have excess $k$.
\end{Def}

In this section, oriented HFPLs of excess 0 and 1 are studied. Throughout this section, $\vcenter{\hbox{\includegraphics[width=.014\textwidth]{Expose5}}}$, 
$\vcenter{\hbox{\includegraphics[width=.014\textwidth]{Expose6}}}$, etc.
denote the numbers of occurrences of local configurations of type $\vcenter{\hbox{\includegraphics[width=.014\textwidth]{Expose5}}}$, $\vcenter{\hbox{\includegraphics[width=.014\textwidth]{Expose6}}}$, etc.

\subsection{Hexagonal Knutson-Tao puzzles}\label{Subsec:Hexagonal_Puzzles}
In this subsection, hexagonal Knutson-Tao puzzles are defined. They can be enumerated by Littlewood-Richardson coefficients, what was observed
by A. Knutson (personal conversation). Furthermore, in the next subsection it will be shown that they are in bijection with both ordinary and oriented HFPLs of excess $0$.

\begin{Def}[\cite{KnutsonTao}]
 A \textit{puzzle piece} is defined as one of the following equilateral plane figures with side length $1$ and labelled edges:
 \begin{center}
  \includegraphics[width=.8\textwidth]{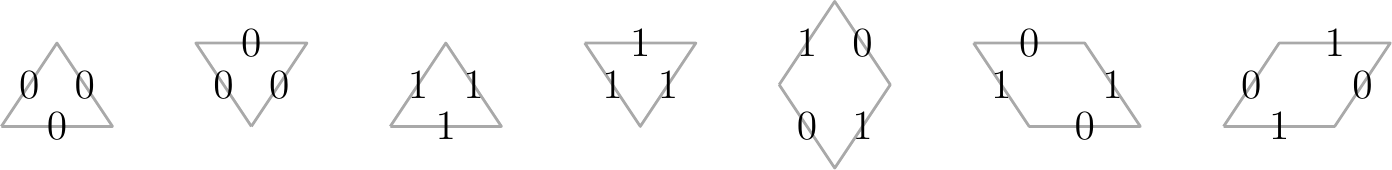} 
	    \end{center}
\end{Def}	    

In the following, the hexagon with vertices $(0,0)$, \mbox{$(\frac{K}{2},\frac{K\sqrt{3}}{2})$}, 
\mbox{$(\frac{K}{2}+L,\frac{K\sqrt{3}}{2})$},\mbox{$(\frac{K+M}{2}+L,\frac{(K-M)\sqrt{3}}{2})$}, \mbox{$(\frac{K+M-N}{2}+L,\frac{(K-M-N)\sqrt{3}}{2})$}  and \mbox{$(\frac{M+N-K}{2},-\frac{(M+N-K)\sqrt{3}}{2})$}
is denoted by $\mathcal{H}_{K,L}^{M,N}$.
A decomposition $P$ of $\mathcal{H}_{K,L}^{M,N}$
into unit triangles and unit rhombi, all edges labelled 0 or 1, such that each region is a puzzle piece is said to be a \textit{hexagonal Knutson-Tao puzzle} of size $(K,L,M,N)$.
Furthermore, a hexagonal Knutson-Tao puzzle is said to have boundary $(\mathsf{l}_\mathsf{T},\mathsf{t},\mathsf{r}_\mathsf{T};\mathsf{r}_\mathsf{B},\mathsf{b},\mathsf{l}_\mathsf{B})$ if the labels of the top left, 
top, top right, bottom right, bottom and bottom left sides of $\mathcal{H}_{K,L}^{M,N}$ are given by $\mathsf{l}_\mathsf{T}$, $\mathsf{t}$, $\mathsf{r}_\mathsf{T}$, $\mathsf{r}_\mathsf{B}$, $\mathsf{b}$ and $\mathsf{l}_\mathsf{B}$ respectively, when read 
from left to right. In Figure~\ref{Fig:Hexagonal_Puzzle}, a hexagonal Knutson-Tao puzzle with boundary $(011,1,1011;110,1,1110)$ is depicted.

\begin{figure}[tbh]
\includegraphics[width=.3\textwidth]{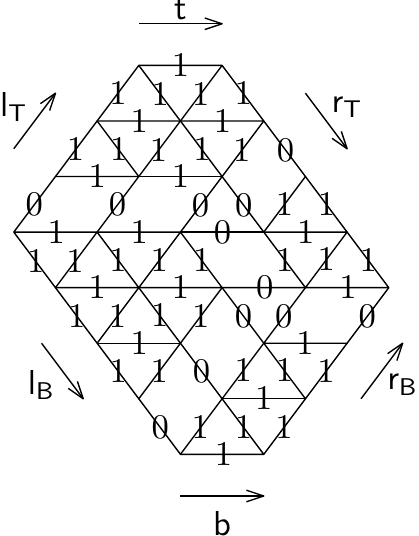}
\caption{A hexagonal Knutson-Tao puzzle with boundary $(011,1,1011;110,1,1110)$.}
\label{Fig:Hexagonal_Puzzle}
\end{figure}

Below, the word of length $n$ consisting solely of ones (resp. zeroes) is denoted by $\textbf{1}_n$ (resp. $\textbf{0}_n$).

\begin{Prop}\label{Prop:Enumeration_Hexagonal_Puzzles}
The number of hexagonal Knutson-Tao puzzles with boundary 
$(\mathsf{l}_\mathsf{T},\mathsf{t},\mathsf{r}_\mathsf{T};\mathsf{r}_\mathsf{B},\mathsf{b},\mathsf{l}_\mathsf{B})$ is given by the Littlewood-Richardson coefficient 
\begin{equation}
c_{\lambda(\textbf{0}_{\vert\mathsf{l}_\mathsf{B}\vert_0}\,\textbf{1}_{\vert\mathsf{l}_\mathsf{B}\vert_1}\,\mathsf{l}_\mathsf{T}\,\mathsf{t}),
\lambda(\textbf{0}_{\vert\mathsf{t}\vert_0}\,\textbf{1}_{\vert\mathsf{t}\vert_1}\,
\mathsf{r}_\mathsf{T}\,\textbf{0}_{\vert\mathsf{r}_\mathsf{B}\vert_0}\,\textbf{1}_{\vert\mathsf{r}_\mathsf{B}\vert_1})}^{\lambda(\mathsf{l}_\mathsf{B}\,\mathsf{b}\,\mathsf{r}_\mathsf{B})}.
\label{Eq:Enumeration_Hexagonal_Puzzles}\end{equation}
\end{Prop}

\begin{proof}
Triangular Knutson-Tao puzzles with boundary
$(\textbf{0}_{\vert\mathsf{l}_\mathsf{B}\vert_0}\,\textbf{1}_{\vert\mathsf{l}_\mathsf{B}\vert_1}\,\mathsf{l}_\mathsf{T}\,\mathsf{t},\textbf{0}_{\vert\mathsf{t}\vert_0}\,\textbf{1}_{\vert\mathsf{t}\vert_1}\,
\mathsf{r}_\mathsf{T}\,\textbf{0}_{\vert\mathsf{r}_\mathsf{B}\vert_0}\,\textbf{1}_{\vert\mathsf{r}_\mathsf{B}\vert_1};
\mathsf{l}_\mathsf{B}\,\mathsf{b}\,\mathsf{r}_\mathsf{B})$ are enumerated by the Littlewood-Richardson coefficient in (\ref{Eq:Enumeration_Hexagonal_Puzzles}) as it is shown in \cite{KnutsonTao}.
Now, a triangular Knutson-Tao puzzle $P$ with boundary 
$(\textbf{0}_{\vert\mathsf{l}_\mathsf{B}\vert_0}\,\textbf{1}_{\vert\mathsf{l}_\mathsf{B}\vert_1}\,\mathsf{l}_\mathsf{T}\,\mathsf{t},\textbf{0}_{\vert\mathsf{t}\vert_0}\,\textbf{1}_{\vert\mathsf{t}\vert_1}\,
\mathsf{r}_\mathsf{T}\,\textbf{0}_{\vert\mathsf{r}_\mathsf{B}\vert_0}\,\textbf{1}_{\vert\mathsf{r}_\mathsf{B}\vert_1};
\mathsf{l}_\mathsf{B}\,\mathsf{b}\,\mathsf{r}_\mathsf{B})$ decomposes into a hexagonal Knutson-Tao puzzle $P'$ of size $(K,L;M,N)$ with boundary 
$(\mathsf{l}_\mathsf{T},\mathsf{t},\mathsf{r}_\mathsf{T};\mathsf{r}_\mathsf{B},\mathsf{b},\mathsf{l}_\mathsf{B})$, the unique triangular 
Knutson-Tao puzzle with boundary $(\mathsf{t},\textbf{0}_{\vert\mathsf{t}\vert_0}
\textbf{1}_{\vert\mathsf{t}\vert_1};\mathsf{t})$, which is attached to $P'$ alongside the top side of $\mathcal{H}_{K,L}^{M,N}$, 
the unique triangular Knutson-Tao puzzle with boundary 
$(\mathsf{r}_\mathsf{B},\textbf{0}_{\vert\mathsf{r}_\mathsf{B}\vert_0}\,\textbf{1}_{\vert\mathsf{r}_\mathsf{B}\vert_1};
\mathsf{r}_\mathsf{B})$, which is attached to $P'$ alongside the bottom right side of $\mathcal{H}_{K,L}^{M,N}$ and the unique triangular Knutson-Tao puzzle with boundary 
$(\textbf{0}_{\vert\mathsf{l}_\mathsf{B}\vert_0}\,\textbf{1}_{\vert\mathsf{l}_\mathsf{B}\vert_1},\mathsf{l}_\mathsf{B};\mathsf{l}_\mathsf{B})$, which is attached to $P'$
alongside the bottom left side of $\mathcal{H}_{K,L}^{M,N}$. The decomposition for a particular case can be seen in Figure~\ref{Fig:Enumeration_Hexagonal_Puzzles}. 
That the three triangular Knutson-Tao puzzles arising in the decomposition are unique is proven in \cite{KnutsonTao}.
So, by mapping $P$ to $P'$ 
a bijection between the set of triangular Knutson-Tao puzzles with boundary
$(\textbf{0}_{\vert\mathsf{l}_\mathsf{B}\vert_0}\,\textbf{1}_{\vert\mathsf{l}_\mathsf{B}\vert_1}\,\mathsf{l}_\mathsf{T}\,\mathsf{t},\textbf{0}_{\vert\mathsf{t}\vert_0}\,\textbf{1}_{\vert\mathsf{t}\vert_1}\,
\mathsf{r}_\mathsf{T}\,\textbf{0}_{\vert\mathsf{r}_\mathsf{B}\vert_0}\,\textbf{1}_{\vert\mathsf{r}_\mathsf{B}\vert_1};
\mathsf{l}_\mathsf{B}\,\mathsf{b}\,\mathsf{r}_\mathsf{B})$ and hexagonal Knutson-Tao puzzles with boundary 
$(\mathsf{l}_\mathsf{T},\mathsf{t},\mathsf{r}_\mathsf{T};\mathsf{r}_\mathsf{B},\mathsf{b},\mathsf{l}_\mathsf{B})$ is obtained.
\end{proof}

In Figure~\ref{Fig:Enumeration_Hexagonal_Puzzles}, the triangular Knutson-Tao puzzle, that corresponds to the hexagonal Knutson-Tao puzzle depicted in Figure~\ref{Fig:Hexagonal_Puzzle}, is pictured. 

\begin{figure}[tbh]
\includegraphics[width=.5\textwidth]{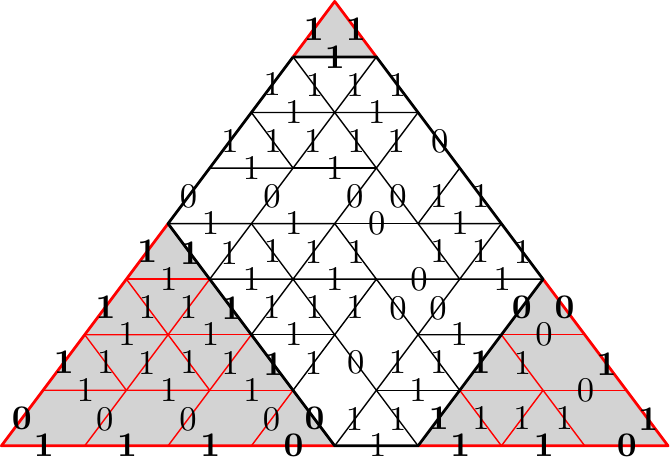}
\caption{The triangular Knutson-Tao puzzle corresponding to the hexagonal one depicted in Figure~\ref{Fig:Hexagonal_Puzzle}.}
\label{Fig:Enumeration_Hexagonal_Puzzles}
\end{figure}

\subsection{Configurations of excess 0}

In this subsection, oriented and ordinary HFPLs of excess $0$ are regarded. 
By Theorem~\ref{Thm:HFPL_Comb_Int}, oriented HFPLs of excess 0 can be characterized as follows:

\begin{Cor}\label{HFPL_Excess_0_Characterization} An oriented HFPL $f$ is of excess $0$ if and only if none 
of the following four configurations occurs in $f$:
\hspace{2cm}$\vcenter{\hbox{\includegraphics[width=.012\textwidth]{Expose5}}}$\hspace{1cm}$\vcenter{\hbox{\includegraphics[width=.012\textwidth]{Expose6}}}$\hspace{1cm}  
$\vcenter{\hbox{\includegraphics[width=.1\textwidth]{Expose7}}}$\hspace{1cm}$\vcenter{\hbox{\includegraphics[width=.1\textwidth]{Expose8}}}$.
\end{Cor}

The characterization above and Theorem~\ref{Thm:NecCondHFPL}(3) imply the following properties of an oriented HFPL of excess $0$.

\begin{Prop}\label{Prop:Properties_HFPL_excess_0}
\begin{enumerate}
 \item An oriented HFPL of excess $0$ contains neither a path joining two vertices in $\mathcal{B}$, that is oriented from right to left, nor a path joining two vertices in $\mathcal{T}$, that is oriented 
 from right to left.
 \item The weight of an oriented HFPL of excess $0$ is 1.
 \item An oriented HFPL of excess $0$ does not contain closed paths.
\end{enumerate}
In particular, $\overrightarrow{h}_{\mathsf{l}_\mathsf{T},\mathsf{t},\mathsf{r}_\mathsf{T}}^{\mathsf{r}_\mathsf{B},\mathsf{b},\mathsf{l}_\mathsf{B}}(q)=
h_{\mathsf{l}_\mathsf{T},\mathsf{t},\mathsf{r}_\mathsf{T}}^{\mathsf{r}_\mathsf{B},\mathsf{b},\mathsf{l}_\mathsf{B}}$ if 
$exc(\mathsf{l}_\mathsf{T},\mathsf{t},\mathsf{r}_\mathsf{T};\mathsf{r}_\mathsf{B},\mathsf{b},\mathsf{l}_\mathsf{B})=0$.
\end{Prop}

The previous proposition generalizes Proposition~5.3 in \cite{TFPL} and also Lemma~13 in \cite{Nadeau2}. In \cite{ZinnJustin}, \cite{Nadeau2} and in \cite{TFPL}, it is shown that odinary respectively oriented TFPLs of 
excess $0$ are in bijection with (triangular) Knutson-Tao puzzles.
The bijection in \cite{TFPL} between oriented TFPLs of excess $0$ and triangular Knutson-Tao puzzles naturally extends to a bijection between
oriented HFPLs of excess 0 and hexagonal Knutson-Tao puzzles. 
In Figure~\ref{Fig:HFPL_Excess_0}, the oriented HFPL of excess $0$ corresponding to the hexagonal Knutson-Tao puzzle depicted in Figure~\ref{Fig:Hexagonal_Puzzle} is given. 

\begin{figure}[tbh]
\includegraphics[width=.45\textwidth]{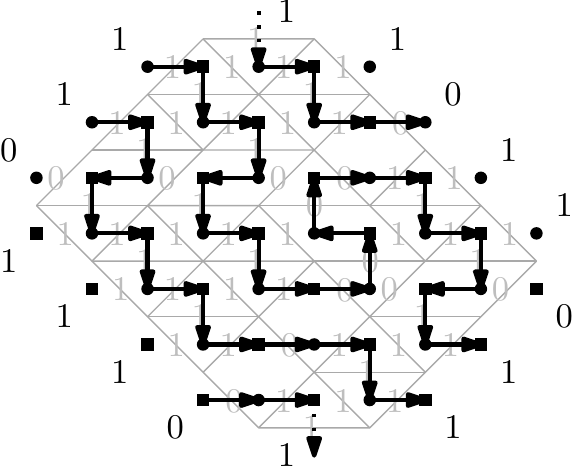}
\caption{The oriented HFPL of excess 0 with boundary $(011,1,1011;110,1,1110)$ that corresponds to the hexagonal Knutson-Tao puzzle with boundary $(011,1,1011;110,1,1110)$ depicted in Figure~\ref{Fig:Hexagonal_Puzzle}.}
\label{Fig:HFPL_Excess_0}
\end{figure}

\begin{Thm}\label{Thm:Bij_HFPL_Exc_0_KT}
Let $(\mathsf{l}_\mathsf{T},\mathsf{t},\mathsf{r}_\mathsf{T};\mathsf{r}_\mathsf{B},\mathsf{b},\mathsf{l}_\mathsf{B})$ be a sextuple of words of length $(K,L,M;N,K+L-N,M+N-K)$ respectively such that 
$exc(\mathsf{l}_\mathsf{T},\mathsf{t},\mathsf{r}_\mathsf{T};\mathsf{r}_\mathsf{B},\mathsf{b},\mathsf{l}_\mathsf{B})=0$. 
Then,
\begin{equation}\overrightarrow{h}_{\mathsf{l}_\mathsf{T},\mathsf{t},\mathsf{r}_\mathsf{T}}^{\mathsf{r}_\mathsf{B},\mathsf{b},\mathsf{l}_\mathsf{B}}=
c_{\lambda(\textbf{0}_{\vert\mathsf{l}_\mathsf{B}\vert_0}\,\textbf{1}_{\vert\mathsf{l}_\mathsf{B}\vert_1}\,\mathsf{l}_\mathsf{T}\,\mathsf{t}),\lambda(\textbf{0}_{\vert\mathsf{t}\vert_0}\,\textbf{1}_{\vert\mathsf{t}\vert_1}
\,\mathsf{r}_\mathsf{T}\,\textbf{0}_{\vert\mathsf{r}_\mathsf{B}\vert_0}\,\textbf{1}_{\vert\mathsf{r}_\mathsf{B}\vert_1})}^{\lambda(\mathsf{l}_\mathsf{B}\,\mathsf{b}\,\mathsf{r}_\mathsf{B})}
.\end{equation}
\end{Thm}

By Proposition~\ref{Prop:Properties_HFPL_excess_0}, it holds $\overrightarrow{h}_{\mathsf{l}_\mathsf{T},\mathsf{t},\mathsf{r}_\mathsf{T}}^{\mathsf{r}_\mathsf{B},\mathsf{b},\mathsf{l}_\mathsf{B}}
=h_{\mathsf{l}_\mathsf{T},\mathsf{t},\mathsf{r}_\mathsf{T}}^{\mathsf{r}_\mathsf{B},\mathsf{b},\mathsf{l}_\mathsf{B}}$
if $exc(\mathsf{l}_\mathsf{T},\mathsf{t},\mathsf{r}_\mathsf{T};\mathsf{r}_\mathsf{B},\mathsf{b},\mathsf{l}_\mathsf{B})=0$. Thus, as an immediate consequence of Theorem~\ref{Thm:Bij_HFPL_Exc_0_KT} one obtains:

\begin{Cor}\label{Cor:Enumeration_HFPL_Excess_0} Let $(\mathsf{l}_\mathsf{T}, \mathsf{t}, \mathsf{r}_\mathsf{T}; \mathsf{r}_\mathsf{B}, \mathsf{b},\mathsf{l}_\mathsf{B})$ be as in 
Theorem~\ref{Thm:Bij_HFPL_Exc_0_KT}. Then,
\begin{equation}
 h_{\mathsf{l}_\mathsf{T},\mathsf{t},\mathsf{r}_\mathsf{T}}^{\mathsf{r}_\mathsf{B},\mathsf{b},\mathsf{l}_\mathsf{B}}=
 c_{\lambda(\textbf{0}_{\vert\mathsf{l}_\mathsf{B}\vert_0}\,\textbf{1}_{\vert\mathsf{l}_\mathsf{B}\vert_1}\,\mathsf{l}_\mathsf{T}\,\mathsf{t}),
 \lambda(\textbf{0}_{\vert\mathsf{t}\vert_0}\,\textbf{1}_{\vert\mathsf{t}\vert_1}\,
 \mathsf{r}_\mathsf{T}\,\textbf{0}_{\vert\mathsf{r}_\mathsf{B}\vert_0}\,\textbf{1}_{\vert\mathsf{r}_\mathsf{B}\vert_1})}^{\lambda(\mathsf{l}_\mathsf{B}\,\mathsf{b}\,\mathsf{r}_\mathsf{B})}. 
\end{equation}
\end{Cor}

\subsection{Configurations of excess 1}

In this subsection, the number of oriented respectively ordinary HFPLs of excess $1$ is expressed in terms of Littlewood-Richardson coefficients. Here, no proofs are given
because they are analogous to the proofs in \cite[Section~6]{TFPL}. The following characterization of oriented HFPLs of excess 1 is an immediate consequence of Theorem~\ref{Thm:HFPL_Comb_Int}.

\begin{Prop}\label{Prop:Characterization_HFPL_excess_1}
An oriented HFPL has excess $1$ if and only if there is one local configuration among the first four in the list below that appears precisely once, whereas the other four configurations in the
list do not appear at all.
\begin{center}
$\vcenter{\hbox{\includegraphics[width=.014\textwidth]{Expose5}}}$ \hspace{.4cm} $\vcenter{\hbox{\includegraphics[width=.014\textwidth]{Expose6}}}$ \hspace{.4cm}
$\vcenter{\hbox{\includegraphics[width=.1\textwidth]{Expose7}}}$ \hspace{.4cm} $\vcenter{\hbox{\includegraphics[width=.1\textwidth]{Expose8}}}$ \hspace{.4cm} $\arrowvert$
\hspace{.4cm} $\vcenter{\hbox{\includegraphics[width=.06\textwidth]{Expose9}}}$ \hspace{.4cm}
$\vcenter{\hbox{\includegraphics[width=.06\textwidth]{Expose10}}}$ \hspace{.4cm} $\vcenter{\hbox{\includegraphics[width=.06\textwidth]{Expose11}}}$ \hspace{.4cm}
		$\vcenter{\hbox{\includegraphics[width=.06\textwidth]{Expose12}}}$
\end{center}
\end{Prop}

By Proposition~\ref{Prop:Characterization_HFPL_excess_1} oriented HFPLs of excess 1 resemble oriented HFPLs of excess 0 but with one ``defect''.
To enumerate oriented HFPLs of excess $1$, the ``defect'' is moved to the boundary of the oriented HFPL of excess $1$ by fixed rules. These rules are the same as in \cite[Section~6]{TFPL} for oriented TFPLs 
of excess $1$. Once on the boundary of the oriented HFPL of excess $1$, the ``defect'' can be deleted and one obtains an oriented HFPL of excess $0$. This is how the enumeration of oriented HFPLs of excess $1$
can be reduced to the enumeration of HFPLs of excess $0$. The resulting expression is stated in Theorem~\ref{Thm:Enumeration_Excess_1_HFPL}(1).

\begin{Def}
Given two words $\omega$ and $\omega^+$ it is written $\omega\longrightarrow\omega^+$ if $\omega=\omega_L01\omega_R$ and $\omega^+=\omega_L10\omega_R$.
Additionally, if $\omega\longrightarrow\omega^+$, then $L_i(\omega,\omega^+)=\vert\omega_L\vert_i$ respectively $R_i(\omega,\omega^+)=\vert\omega_R\vert_i$ for $i=0,1$ and $L(\omega,\omega^+)=L_0(\omega,\omega^+)+L_1(\omega,\omega^+)+1$ respectively
$R(\omega,\omega^+)=R_0(\omega,\omega^+)+R_1(\omega,\omega^+)+1$. 
\end{Def}
Below, it is written $c_{u,v}^w$ instead of $c_{\lambda(u),\lambda(v)}^{\lambda(w)}$ for words $u,v,w$.

\begin{Thm}\label{Thm:Enumeration_Excess_1_HFPL} Let $(\mathsf{l}_\mathsf{T},\mathsf{t},\mathsf{r}_\mathsf{T};\mathsf{r}_\mathsf{B},\mathsf{b},\mathsf{l}_\mathsf{B})$ be a sextuple of words of length 
$(K,L,N,N,K+L-N,M+N-K)$ respectively such that $exc(\mathsf{l}_\mathsf{T},\mathsf{t},\mathsf{r}_\mathsf{T};\mathsf{r}_\mathsf{B},\mathsf{b},\mathsf{l}_\mathsf{B})=1$. 
\begin{enumerate}
\item The number of oriented HFPLs with boundary $(\mathsf{l}_\mathsf{T},\mathsf{t},\mathsf{r}_\mathsf{T};\mathsf{r}_\mathsf{B},\mathsf{b},\mathsf{l}_\mathsf{B})$ is 
\begin{align} & \sum\limits_{\mathsf{l}_\mathsf{T}^+:\mathsf{l}_\mathsf{T}\longrightarrow\mathsf{l}_\mathsf{T}^+}
(\vert\mathsf{l}_\mathsf{T}\vert_1+\vert\mathsf{t}\vert_1+L_1(\mathsf{l}_\mathsf{T},\mathsf{l}_\mathsf{T}^+))
c_{\textbf{0}_{\vert\mathsf{l}_\mathsf{B}\vert_0}\,\textbf{1}_{\vert\mathsf{l}_\mathsf{B}\vert_1}\,\mathsf{l}_\mathsf{T}^+\,\mathsf{t},
\textbf{0}_{\vert\mathsf{t}\vert_0}\,\textbf{1}_{\vert\mathsf{t}\vert_1}\,\mathsf{r}_\mathsf{T}\,\textbf{0}_{\vert\mathsf{r}_\mathsf{B}\vert_0}\,\textbf{1}_{\vert\mathsf{r}_\mathsf{B}\vert_1}}
^{\mathsf{l}_\mathsf{B}\,\mathsf{b}\,\mathsf{r}_\mathsf{B}}\notag\\
& +\sum\limits_{\mathsf{t}^+:\mathsf{t}\longrightarrow\mathsf{t}^+}2(\vert\mathsf{l}_\mathsf{T}\vert_1+L_1(\mathsf{t},\mathsf{t}^+))
c_{\textbf{0}_{\vert\mathsf{l}_\mathsf{B}\vert_0}\,\textbf{1}_{\vert\mathsf{l}_\mathsf{B}\vert_1}\,\mathsf{l}_\mathsf{T}\,\mathsf{t}^+,\textbf{0}_{\vert\mathsf{t}\vert_0}\,\textbf{1}_{\vert\mathsf{t}\vert_1}\,
\mathsf{r}_\mathsf{T}\,\textbf{0}_{\vert\mathsf{r}_\mathsf{B}\vert_0}\,\textbf{1}_{\vert\mathsf{r}_\mathsf{B}\vert_1}}^{\mathsf{l}_\mathsf{B}\,\mathsf{b}\,\mathsf{r}_\mathsf{B}}\notag\\
& +\sum\limits_{\mathsf{r}_\mathsf{T}^+:\mathsf{r}_\mathsf{T}\longrightarrow\mathsf{r}_\mathsf{T}^+}(L+\vert\mathsf{t}\vert_1+L(\mathsf{r}_\mathsf{T},\mathsf{r}_\mathsf{T}^+)+
L_1(\mathsf{r}_\mathsf{T},\mathsf{r}_\mathsf{T}^+)+1)
c_{\textbf{0}_{\vert\mathsf{l}_\mathsf{B}\vert_0}\,\textbf{1}_{\vert\mathsf{l}_\mathsf{B}\vert_1}\,\mathsf{l}_\mathsf{T}\,\mathsf{t},\textbf{0}_{\vert\mathsf{t}\vert_0}\,\textbf{1}_{\vert\mathsf{t}\vert_1}\,
\mathsf{r}_\mathsf{T}^+\,\textbf{0}_{\vert\mathsf{r}_\mathsf{B}\vert_0}\,\textbf{1}_{\vert\mathsf{r}_\mathsf{B}\vert_1}}^{\mathsf{l}_\mathsf{B}\,\mathsf{b}\,\mathsf{r}_\mathsf{B}}\notag\\
& +\sum\limits_{\mathsf{r}_\mathsf{B}^-:\mathsf{r}_\mathsf{B}^-\longrightarrow\mathsf{r}_\mathsf{B}}(L+M+\vert\mathsf{l}_\mathsf{B}\vert_1+1-\vert\mathsf{r}_\mathsf{B}\vert_1-\vert\mathsf{b}\vert_1-
L_1(\mathsf{r}_\mathsf{B}^-,\mathsf{r}_\mathsf{B}))
c_{\textbf{0}_{\vert\mathsf{l}_\mathsf{B}\vert_0}\,\textbf{1}_{\vert\mathsf{l}_\mathsf{B}\vert_1}\,\mathsf{l}_\mathsf{T}\,\mathsf{t},\textbf{0}_{\vert\mathsf{t}\vert_0}\,\textbf{1}_{\vert\mathsf{t}\vert_1}\,
\mathsf{r}_\mathsf{T}\,\textbf{0}_{\vert\mathsf{r}_\mathsf{B}\vert_0}\,\textbf{1}_{\vert\mathsf{r}_\mathsf{B}\vert_1}}^{\mathsf{l}_\mathsf{B}\,\mathsf{b}\,\mathsf{r}_\mathsf{B}^-}\notag\\
& - \sum\limits_{\mathsf{b}^-:\mathsf{b}^-\longrightarrow\mathsf{b}}2L_1(\mathsf{b}^-,\mathsf{b})
c_{\textbf{0}_{\vert\mathsf{l}_\mathsf{B}\vert_0}\,\textbf{1}_{\vert\mathsf{l}_\mathsf{B}\vert_1}\,\mathsf{l}_\mathsf{T}\,\mathsf{t},\textbf{0}_{\vert\mathsf{t}\vert_0}\,\textbf{1}_{\vert\mathsf{t}\vert_1}\,
\mathsf{r}_\mathsf{T}\,\textbf{0}_{\vert\mathsf{r}_\mathsf{B}\vert_0}\,\textbf{1}_{\vert\mathsf{r}_\mathsf{B}\vert_1}}^{\mathsf{l}_\mathsf{B}\,\mathsf{b}^-\,\mathsf{r}_\mathsf{B}}\notag\\
& +\sum\limits_{\mathsf{l}_\mathsf{B}^-:\mathsf{l}_\mathsf{B}^-\longrightarrow\mathsf{l}_\mathsf{B}}(\vert\mathsf{l}_\mathsf{T}\vert_1+\vert\mathsf{t}\vert_1-L(\mathsf{l}_\mathsf{B}^-,\mathsf{l}_\mathsf{B})-
L_1(\mathsf{l}_\mathsf{B}^-,\mathsf{l}_\mathsf{B}))
c_{\textbf{0}_{\vert\mathsf{l}_\mathsf{B}\vert_0}\,\textbf{1}_{\vert\mathsf{l}_\mathsf{B}\vert_1}\,\mathsf{l}_\mathsf{T}\,\mathsf{t},\textbf{0}_{\vert\mathsf{t}\vert_0}\,\textbf{1}_{\vert\mathsf{t}\vert_1}\,
\mathsf{r}_\mathsf{T}\,\textbf{0}_{\vert\mathsf{r}_\mathsf{B}\vert_0}\,\textbf{1}_{\vert\mathsf{r}_\mathsf{B}\vert_1}}^{\mathsf{l}_\mathsf{B}^-\,\mathsf{b}\,\mathsf{r}_\mathsf{B}}.\notag
\end{align}
\item The weighted enumeration of oriented HFPLs with boundary $(\mathsf{l}_\mathsf{T},\mathsf{t},\mathsf{r}_\mathsf{T};\mathsf{r}_\mathsf{B},\mathsf{b},\mathsf{l}_\mathsf{B})$ is 
\begin{align}
 & \sum\limits_{\mathsf{l}_\mathsf{T}^+:\mathsf{l}_\mathsf{T}\longrightarrow\mathsf{l}_\mathsf{T}^+}(R_1(\mathsf{l}_\mathsf{T},\mathsf{l}_\mathsf{T}^+)+\vert\mathsf{t}\vert_1+1+(q+q^{-1})
 L_1(\mathsf{l}_\mathsf{T},\mathsf{l}_\mathsf{T}^+))
 c_{\textbf{0}_{\vert\mathsf{l}_\mathsf{B}\vert_0}\,\textbf{1}_{\vert\mathsf{l}_\mathsf{B}\vert_1}\,\mathsf{l}_\mathsf{T}^+\,\mathsf{t},\textbf{0}_{\vert\mathsf{t}\vert_0}\,\textbf{1}_{\vert\mathsf{t}\vert_1}\,
 \mathsf{r}_\mathsf{T}\,\textbf{0}_{\vert\mathsf{r}_\mathsf{B}\vert_0}\,\textbf{1}_{\vert\mathsf{r}_\mathsf{B}\vert_1}}^{\mathsf{l}_\mathsf{B}\,\mathsf{b}\,\mathsf{r}_\mathsf{B}}\notag\\
& +\sum\limits_{\mathsf{t}^+:\mathsf{t}\longrightarrow\mathsf{t}^+}((q+q^{-1})(\vert\mathsf{l}_\mathsf{T}\vert_1+L_1(\mathsf{t},\mathsf{t}^+))-q)
c_{\textbf{0}_{\vert\mathsf{l}_\mathsf{B}\vert_0}\,\textbf{1}_{\vert\mathsf{l}_\mathsf{B}\vert_1}\,\mathsf{l}_\mathsf{T}\,\mathsf{t}^+,\textbf{0}_{\vert\mathsf{t}\vert_0}\,\textbf{1}_{\vert\mathsf{t}\vert_1}\,
\mathsf{r}_\mathsf{T}\,\textbf{0}_{\vert\mathsf{r}_\mathsf{B}\vert_0}\,\textbf{1}_{\vert\mathsf{r}_\mathsf{B}\vert_1}}^{\mathsf{l}_\mathsf{B}\,\mathsf{b}\,\mathsf{r}_\mathsf{B}}
\notag\\
& +\sum\limits_{\mathsf{r}_\mathsf{T}^+:\mathsf{r}_\mathsf{T}\longrightarrow\mathsf{r}_\mathsf{T}^+}(\vert\mathsf{t}\vert_0+1+L_0(\mathsf{r}_\mathsf{T},\mathsf{r}_\mathsf{T}^+)+(q+q^{-1})(\vert\mathsf{t}\vert_1+
L_1(\mathsf{r}_\mathsf{T},\mathsf{r}_\mathsf{T}^+)))
c_{\textbf{0}_{\vert\mathsf{l}_\mathsf{B}\vert_0}\,\textbf{1}_{\vert\mathsf{l}_\mathsf{B}\vert_1}\,\mathsf{l}_\mathsf{T}\,\mathsf{t},\textbf{0}_{\vert\mathsf{t}\vert_0}\,\textbf{1}_{\vert\mathsf{t}\vert_1}\,
\mathsf{r}_\mathsf{T}^+\,\textbf{0}_{\vert\mathsf{r}_\mathsf{B}\vert_0}\,\textbf{1}_{\vert\mathsf{r}_\mathsf{B}\vert_1}}^{\mathsf{l}_\mathsf{B}\,\mathsf{b}\,\mathsf{r}_\mathsf{B}}\notag\\
& + \sum\limits_{\mathsf{r}_\mathsf{B}^-:\mathsf{r}_\mathsf{B}^-\longrightarrow\mathsf{r}_\mathsf{B}}(\vert\mathsf{t}\vert_0+\vert\mathsf{r}_\mathsf{T}\vert_0-R_1(\mathsf{r}_\mathsf{B}^-,
\mathsf{r}_\mathsf{B})+(q+q^{-1})(\vert\mathsf{l}_\mathsf{B}\vert_1-L_1(\mathsf{r}_\mathsf{B}^-,\mathsf{r}_\mathsf{B})))
c_{\textbf{0}_{\vert\mathsf{l}_\mathsf{B}\vert_0}\,\textbf{1}_{\vert\mathsf{l}_\mathsf{B}\vert_1}\,\mathsf{l}_\mathsf{T}\mathsf{t},\textbf{0}_{\vert\mathsf{t}\vert_0}\,\textbf{1}_{\vert\mathsf{t}\vert_1}\,
\mathsf{r}_\mathsf{T}\,\textbf{0}_{\vert\mathsf{r}_\mathsf{B}\vert_0}\,\textbf{1}_{\vert\mathsf{r}_\mathsf{B}\vert_1}}^{\mathsf{l}_\mathsf{B}\,\mathsf{b}\,\mathsf{r}_\mathsf{B}^-}\notag\\
& - \sum\limits_{\mathsf{b}^-:\mathsf{b}^-\longrightarrow\mathsf{b}}((q+q^{-1})L_1(\mathsf{b}^-,\mathsf{b})+q)
c_{\textbf{0}_{\vert\mathsf{l}_\mathsf{B}\vert_0}\,\textbf{1}_{\vert\mathsf{l}_\mathsf{B}\vert_1}\,\mathsf{l}_\mathsf{T}\,\mathsf{t},\textbf{0}_{\vert\mathsf{t}\vert_0}\,\textbf{1}_{\vert\mathsf{t}\vert_1}\,
\mathsf{r}_\mathsf{T}\,\textbf{0}_{\vert\mathsf{r}_\mathsf{B}\vert_0}\,\textbf{1}_{\vert\mathsf{r}_\mathsf{B}\vert_1}}^{\mathsf{l}_\mathsf{B}\,\mathsf{b}^-\,\mathsf{r}_\mathsf{B}}\notag\\
& +\sum\limits_{\mathsf{l}_\mathsf{B}^-:\mathsf{l}_\mathsf{B}^-\longrightarrow\mathsf{l}_\mathsf{B}}(\vert\mathsf{l}_\mathsf{T}\vert_1+\vert\mathsf{t}\vert_1-L_0(\mathsf{l}_\mathsf{B}^-,\mathsf{l}_\mathsf{B})-
(q+q^{-1})L_1(\mathsf{l}_\mathsf{B}^-,\mathsf{l}_\mathsf{B}))
c_{\textbf{0}_{\vert\mathsf{l}_\mathsf{B}\vert_0}\,\textbf{1}_{\vert\mathsf{l}_\mathsf{B}\vert_1}\,\mathsf{l}_\mathsf{T}\,\mathsf{t},\textbf{0}_{\vert\mathsf{t}\vert_0}\,\textbf{1}_{\vert\mathsf{t}\vert_1}\,
\mathsf{r}_\mathsf{T}\,\textbf{0}_{\vert\mathsf{r}_\mathsf{B}\vert_0}\,\textbf{1}_{\vert\mathsf{r}_\mathsf{B}\vert_1}}^{\mathsf{l}_\mathsf{B}^-\,\mathsf{b}\,\mathsf{r}_\mathsf{B}}.
\notag\end{align}
\item The number of HFPLs with boundary $(\mathsf{l}_\mathsf{T},\mathsf{t},\mathsf{r}_\mathsf{T};\mathsf{r}_\mathsf{B},\mathsf{b},\mathsf{l}_\mathsf{B})$ is 
\begin{align}
& \sum\limits_{\mathsf{l}_\mathsf{T}^+:\mathsf{l}_\mathsf{T}\longrightarrow\mathsf{l}_\mathsf{T}^+}(\vert\mathsf{l}_\mathsf{T}\vert_1+\vert\mathsf{t}\vert_1)
c_{\textbf{0}_{\vert\mathsf{l}_\mathsf{B}\vert_0}\,\textbf{1}_{\vert\mathsf{l}_\mathsf{B}\vert_1}\,\mathsf{l}_\mathsf{T}^+\,\mathsf{t},\textbf{0}_{\vert\mathsf{t}\vert_0}\,\textbf{1}_{\vert\mathsf{t}\vert_1}\,
\mathsf{r}_\mathsf{T}\,\textbf{0}_{\vert\mathsf{r}_\mathsf{B}\vert_0}\,\textbf{1}_{\vert\mathsf{r}_\mathsf{B}\vert_1}}^{\mathsf{l}_\mathsf{B}\,\mathsf{b}\,\mathsf{r}_\mathsf{B}}\notag\\
& +\sum\limits_{\mathsf{t}^+:\mathsf{t}\longrightarrow\mathsf{t}^+}(\vert\mathsf{l}_\mathsf{T}\vert_1+L_1(\mathsf{t},\mathsf{t}^+)-1)
c_{\textbf{0}_{\vert\mathsf{l}_\mathsf{B}\vert_0}\,\textbf{1}_{\vert\mathsf{l}_\mathsf{B}\vert_1}\,\mathsf{l}_\mathsf{T}\,\mathsf{t}^+,\textbf{0}_{\vert\mathsf{t}\vert_0}\,\textbf{1}_{\vert\mathsf{t}\vert_1}\,
\mathsf{r}_\mathsf{T}\,\textbf{0}_{\vert\mathsf{r}_\mathsf{B}\vert_0}\,\textbf{1}_{\vert\mathsf{r}_\mathsf{B}\vert_1}}^{\mathsf{l}_\mathsf{B}\,\mathsf{b}\,\mathsf{r}_\mathsf{B}}\notag\\
& +\sum\limits_{\mathsf{r}_\mathsf{T}^+:\mathsf{r}_\mathsf{T}\longrightarrow\mathsf{r}_\mathsf{T}^+}(L+L(\mathsf{r}_\mathsf{T},\mathsf{r}_\mathsf{T}^+))
c_{\textbf{0}_{\vert\mathsf{l}_\mathsf{B}\vert_0}\,\textbf{1}_{\vert\mathsf{l}_\mathsf{B}\vert_1}\,\mathsf{l}_\mathsf{T}\,\mathsf{t},\textbf{0}_{\vert\mathsf{t}\vert_0}\,\textbf{1}_{\vert\mathsf{t}\vert_1}\,
\mathsf{r}_\mathsf{T}^+\,\textbf{0}_{\vert\mathsf{r}_\mathsf{B}\vert_0}\,\textbf{1}_{\vert\mathsf{r}_\mathsf{B}\vert_1}}^{\mathsf{l}_\mathsf{B}\,\mathsf{b}\,\mathsf{r}_\mathsf{B}}\notag\\
& + \sum\limits_{\mathsf{r}_\mathsf{B}^-:\mathsf{r}_\mathsf{B}^-\longrightarrow\mathsf{r}_\mathsf{B}}(\vert\mathsf{t}\vert_0+\vert\mathsf{r}_\mathsf{T}\vert_0+\vert\mathsf{l}_\mathsf{B}\vert_1-
\vert\mathsf{r}_\mathsf{B}\vert_1+1)
c_{\textbf{0}_{\vert\mathsf{l}_\mathsf{B}\vert_0}\,\textbf{1}_{\vert\mathsf{l}_\mathsf{B}\vert_1}\,\mathsf{l}_\mathsf{T}\,\mathsf{t},\textbf{0}_{\vert\mathsf{t}\vert_0}\,\textbf{1}_{\vert\mathsf{t}\vert_1}\,
\mathsf{r}_\mathsf{T}\,\textbf{0}_{\vert\mathsf{r}_\mathsf{B}\vert_0}\,\textbf{1}_{\vert\mathsf{r}_\mathsf{B}\vert_1}}^{\mathsf{l}_\mathsf{B}\,\mathsf{b}\,\mathsf{r}_\mathsf{B}^-}\notag\\
& - \sum\limits_{\mathsf{b}^-:\mathsf{b}^-\longrightarrow\mathsf{b}}L_1(\mathsf{b}^-,\mathsf{b})
c_{\textbf{0}_{\vert\mathsf{l}_\mathsf{B}\vert_0}\,\textbf{1}_{\vert\mathsf{l}_\mathsf{B}\vert_1}\,\mathsf{l}_\mathsf{T}\,\mathsf{t},\textbf{0}_{\vert\mathsf{t}\vert_0}\,\textbf{1}_{\vert\mathsf{t}\vert_1}\,
\mathsf{r}_\mathsf{T}\,\textbf{0}_{\vert\mathsf{r}_\mathsf{B}\vert_0}\,\textbf{1}_{\vert\mathsf{r}_\mathsf{B}\vert_1}}^{\mathsf{l}_\mathsf{B}\,\mathsf{b}^-\,\mathsf{r}_\mathsf{B}}\notag\\
& +\sum\limits_{\mathsf{l}_\mathsf{B}^-:\mathsf{l}_\mathsf{B}^-\longrightarrow\mathsf{l}_\mathsf{B}}(\vert\mathsf{l}_\mathsf{T}\vert_1+\vert\mathsf{t}\vert_1-L(\mathsf{l}_\mathsf{B}^-,\mathsf{l}_\mathsf{B})+1)
c_{\textbf{0}_{\vert\mathsf{l}_\mathsf{B}\vert_0}\,\textbf{1}_{\vert\mathsf{l}_\mathsf{B}\vert_1}\,\mathsf{l}_\mathsf{T}\,\mathsf{t},\textbf{0}_{\vert\mathsf{t}\vert_0}\,\textbf{1}_{\vert\mathsf{t}\vert_1}\,
\mathsf{r}_\mathsf{T}\,\textbf{0}_{\vert\mathsf{r}_\mathsf{B}\vert_0}\,\textbf{1}_{\vert\mathsf{r}_\mathsf{B}\vert_1}}^{\mathsf{l}_\mathsf{B}^-\,\mathsf{b}\,\mathsf{r}_\mathsf{B}}.
\notag\end{align}

\end{enumerate}
\end{Thm}
\begin{proof}
No details about how to derive (1) and (2) are given here because it can be done in the same way as it is done for TFPLs in \cite[Section 6]{TFPL}. Here, it will be focussed on the proof of the third part of the 
theorem. In the case when $exc(\mathsf{l}_\mathsf{T},\mathsf{t},\mathsf{r}_\mathsf{T};\mathsf{r}_\mathsf{B},\mathsf{b},\mathsf{l}_\mathsf{B})=1$, (\ref{Eq:Weighted_Relations_HFPL}) simplifies to
\begin{equation}
\overrightarrow{h}_{\mathsf{l}_\mathsf{T},\mathsf{t},\mathsf{r}_\mathsf{T}}^{\mathsf{r}_\mathsf{B},\mathsf{b},\mathsf{l}_\mathsf{B}}(q)= 
\overline{h}_{\mathsf{l}_\mathsf{T},\mathsf{t},\mathsf{r}_\mathsf{T}}^{\mathsf{r}_\mathsf{B},\mathsf{b},\mathsf{l}_\mathsf{B}}(q)+q^{-1}\sum\limits_{\mathsf{t}^+:\mathsf{t}\longrightarrow\mathsf{t}^+}
\overline{h}_{\mathsf{l}_\mathsf{T},\mathsf{t}^+,\mathsf{r}_\mathsf{T}}^{\mathsf{r}_\mathsf{B},\mathsf{b},\mathsf{l}_\mathsf{B}}(q)+q\sum\limits_{\mathsf{b}^-:\mathsf{b}^-\longrightarrow\mathsf{b}}
\overline{h}_{\mathsf{l}_\mathsf{T},\mathsf{t},\mathsf{r}_\mathsf{T}}^{\mathsf{r}_\mathsf{B},\mathsf{b}^-,\mathsf{l}_\mathsf{B}}(q).
\notag\end{equation}
In the case when $\mathsf{t}\longrightarrow\mathsf{t}^+$ and $\mathsf{b}^-\longrightarrow\mathsf{b}$, it has to hold $exc(\mathsf{l}_\mathsf{T},\mathsf{t}^+,\mathsf{r}_\mathsf{T};\mathsf{r}_\mathsf{B},\mathsf{b},\mathsf{l}_\mathsf{B})=
exc(\mathsf{l}_\mathsf{T},\mathsf{t},\mathsf{r}_\mathsf{T};\mathsf{r}_\mathsf{B},\mathsf{b}^-,\mathsf{l}_\mathsf{B})=0$.
Thus, $\overline{h}_{\mathsf{l}_\mathsf{T},\mathsf{t}^+,\mathsf{r}_\mathsf{T}}^{\mathsf{r}_\mathsf{B},\mathsf{b},\mathsf{l}_\mathsf{B}}(q)=
h_{\mathsf{l}_\mathsf{T},\mathsf{t}^+,\mathsf{r}_\mathsf{T}}^{\mathsf{r}_\mathsf{B},\mathsf{b},\mathsf{l}_\mathsf{B}}$
and 
 $\overline{h}_{\mathsf{l}_\mathsf{T},\mathsf{t},\mathsf{r}_\mathsf{T}}^{\mathsf{r}_\mathsf{B},\mathsf{b}^-,\mathsf{l}_\mathsf{B}}(q)=
h_{\mathsf{l}_\mathsf{T},\mathsf{t},\mathsf{r}_\mathsf{T}}^{\mathsf{r}_\mathsf{B},\mathsf{b}^-,\mathsf{l}_\mathsf{B}}$
by Proposition~\ref{Prop:Properties_HFPL_excess_0}.
Therefore,
\begin{align}
\overline{h}_{\mathsf{l}_\mathsf{T},\mathsf{t},\mathsf{r}_\mathsf{T}}^{\mathsf{r}_\mathsf{B},\mathsf{b},\mathsf{l}_\mathsf{B}}(q) = & 
\overrightarrow{h}_{\mathsf{l}_\mathsf{T},\mathsf{t},\mathsf{r}_\mathsf{T}}^{\mathsf{r}_\mathsf{B},\mathsf{b},\mathsf{l}_\mathsf{B}}(q)
-q^{-1}\sum\limits_{\mathsf{t}\longrightarrow\mathsf{t}^+}c_{\textbf{0}_{\vert\mathsf{l}_\mathsf{B}\vert_0}\,\textbf{1}_{\vert\mathsf{l}_\mathsf{B}\vert_1}\,\mathsf{l}_\mathsf{T}\,
\mathsf{t}^+,\textbf{0}_{\vert\mathsf{t}\vert_0}\,\textbf{1}_{\vert\mathsf{t}\vert_1}\,\mathsf{r}_\mathsf{T}\,\textbf{0}_{\vert\mathsf{r}_\mathsf{B}\vert_0}\,\textbf{1}_{\vert\mathsf{r}_\mathsf{B}\vert_1}}^{
\mathsf{l}_\mathsf{B}\,\mathsf{b}\,\mathsf{r}_\mathsf{B}}\notag\\
& - q\sum\limits_{\mathsf{b}^-\longrightarrow\mathsf{b}} c_{\textbf{0}_{\vert\mathsf{l}_\mathsf{B}\vert_0}\,\textbf{1}_{\vert\mathsf{l}_\mathsf{B}\vert_1}\,\mathsf{l}_\mathsf{T}\,\mathsf{t},
\textbf{0}_{\vert\mathsf{t}\vert_0}\,\textbf{1}_{\vert\mathsf{t}\vert_1}\,\mathsf{r}_\mathsf{T}\,\textbf{0}_{\vert\mathsf{r}_\mathsf{B}\vert_0}\,\textbf{1}_{\vert\mathsf{r}_\mathsf{B}\vert_1}}^{\mathsf{l}_\mathsf{B}\,
\mathsf{b}^-\mathsf{r}_\mathsf{B}}.
\notag\end{align}
By Lemma~\ref{Lemma:HFPL_Weighted_Enumeration_Without_RL} and by Theorem~\ref{Thm:Enumeration_Excess_1_HFPL}(2) identity (3) follows immediately.
\end{proof}

\bibliographystyle{alpha}
\bibliography{WielandTFPL}

\end{document}